\documentclass[10pt, letterpaper]{article}
\usepackage{amssymb,amsmath, amsthm}
\usepackage{graphics, epsfig}
\usepackage[toc,page]{appendix}

\newtheorem{assumption}{Assumption}[section]
\newtheorem{lemma}{Lemma}[section]
\newtheorem{proposition}{Proposition}[section]
\newtheorem{cor}{Corollary}[section]
\newtheorem{theorem}{Theorem}[section]
\newtheorem{remark}{Remark}[section]

\newtheorem{definition}{Definition}[section]

\numberwithin{equation}{section}

\def\mbf{\mathbf}
\def\asto{\overset{a.s.}{\to}}

\def\E{\text{\rm E}}
 
\def\rn{\Re^n}

\def\g{\gamma}

\def\F{\mathcal{F}}

\def\P{\mathcal{P}}

\def\argmin{\mathop{\arg \min}}
\def\argmax{\mathop{\arg \max}}
\def\supp{\mathop {\text{\rm supp}}}
\def\G{\tilde{g}}

\def\U{\mathcal{U}}
\def\Ste{S_o}
\def\Stce{R_o}
\def\St{S}
\def\Stc{R}

\def\sspa{SSP($\bplb$)} 
\def\A{{\bar {\mathcal{N}}}}
\def\B{{\bar {\mathcal{N}}}}

\def\da{\rho}
\def\db{\sigma}
\def\pla{\mu}
\def\plb{\nu}
\def\bpla{\bar{\mu}}
\def\bplb{\bar{\nu}}
\def\ave{\underline}

\def\Ra{\Pi_{1, \text{\rm \tiny SR}}} 
\def\Rb{\Pi_{2, \text{\rm \tiny SR}}}
 
\def\SR{\Pi_{\text{\rm \tiny SR}}} 
\def\SD{\Pi_{\text{\rm \tiny SD}}} 

\def\moveup{\vspace*{-0.1cm}}

\def\nbpla{\bar \mu}
\def\nbplb{\bar \nu}
\def\npla{\mu}
\def\nplb{\nu}
\def\nDa{D_1} 
\def\nDb{D_2}

\begin{document} 
\title{\bigskip
Stochastic Shortest Path Games and Q-Learning\thanks{This work was supported by the Air Force Grant FA9550-10-1-0412.}}
 \author{Huizhen Yu%
 \thanks{Huizhen Yu was with the Laboratory for Information and Decision Systems (LIDS), MIT, and she is now with the Department of Computing Science, University of Alberta.
        {\tt\small janey.hzyu@gmail.edu}}}
\date{}
\maketitle

\begin{abstract}
We consider a class of  two-player zero-sum stochastic games with finite state and compact control spaces, which we call stochastic shortest path (SSP) games. They are undiscounted total cost stochastic dynamic games that have a cost-free termination state.
Exploiting the close connection of these games to single-player SSP problems,
we introduce novel model conditions under which we show that the SSP games have strong optimality properties, including the existence of a unique solution to the dynamic programming equation, the existence of optimal stationary policies, and the convergence of value and policy iteration.
We then focus on finite state and control SSP games and the classical Q-learning algorithm for computing the value function. Q-learning is a model-free, asynchronous stochastic iterative algorithm. By the theory of stochastic approximation involving monotone nonexpansive mappings,  it is known to converge when its associated dynamic programming equation has a unique solution and its iterates are bounded with probability one. For the SSP case, as the main result of this paper, we prove the boundedness of the Q-learning iterates under our proposed model conditions, thereby establishing completely the convergence of Q-learning for a broad class of total cost finite-space stochastic games. 
\end{abstract}

\thispagestyle{myheadings}
\markboth{}{ \large  LIDS REPORT 2875}
\setlength{\unitlength}{1cm}
\begin{picture}(0,0)(0,0)
\put(-0.52,12.0){\fontsize{12}{14} \selectfont \it Dec 2011}
\put(-0.46,11.5){\fontsize{12}{14}\selectfont \it Revised Apr 2014}
\end{picture}

\clearpage
\tableofcontents
\clearpage

\section{Introduction}

In this paper we consider two-player zero-sum stochastic dynamic games under the undiscounted, total cost criterion, and we focus on those games that have a finite state space and a cost-free termination state. Our interest is in using a well-known model-free stochastic approximation algorithm, the Q-learning algorithm, for computing the value of a game when the control spaces of both players are finite. The main purpose of this paper is to show that there is a broad class of total cost games with desirable optimality properties for which the Q-learning algorithm converges in a totally asynchronous setting under fairly mild conditions.
  
Zero-sum stochastic games were first introduced by Shapley \cite{sha53} for the discounted cost criteria. 
Since then there have been extensive research on undiscounted stochastic games, 
including games with the limiting average cost criterion, first considered by Gillette~\cite{game_Gi57} and developed in the seminal works~\cite{game_BlF68,game_BeK76,game_MeN81}, and games with total cost and related criteria \cite{game_ThV87, game_Fe80, game_No85,game_No99}. 
(We refer readers to the excellent book by Filar and Vrieze~\cite{game_fv97} for historical and contemporary developments on stochastic games.)
A general formulation of total cost games, when one-stage costs can be positive or negative-valued,
was first proposed and analyzed by Thuijsman and Vrieze~\cite{game_ThV87}. 
In this and their subsequent works (see the survey by Thuijsman and Vrieze~\cite{game_ThV98} and also Filar and Vrieze~\cite[Chapter~4]{game_fv97}), they established important existence results for finite state and control total cost games. They showed that for a total cost game to have a finite value function, a sufficient condition is that the corresponding average-cost game has the value zero and both players posses stationary average-cost optimal policies.  They also showed that for a total cost game to have not only a finite value function but also stationary optimal policies for both players, a necessary and sufficient condition is that a certain system of functional equations have solutions.

In this paper we will focus on a subset of the total cost games of the latter kind. In addition to having a value and stationary optimal policies,
the SSP games we consider also have the property that their associated Bellman equation has a unique solution. This property relates to the convergence of value iteration and is essential for the Q-learning algorithm we are interested in. Among the total cost games satisfying Thuijsman and Vrieze's necessary and sufficient conditions mentioned earlier, the ones that will be excluded from our consideration are, briefly speaking, those games in which from some initial state, both players can play some stationary optimal policies (and incur zero average cost) without ever reaching the termination state. (We will discuss in Section~\ref{sec-ssp-remarks} some examples of such games.)

To delineate a subset of SSP games with desirable properties, we will specify conditions on the model of the games, and we will do so in the broader context of games with compact control sets and semi-continuous one-stage costs, which include finite-control games as special cases.  (A finite-control game can be viewed as a game with compact control sets, where controls correspond to randomized decision rules of each player.) In the context of total cost compact-control games, there are several earlier works \cite{KuC69, KuS81, sspgame_pt99}, and the one by Patek and Bertsekas \cite{sspgame_pt99} is most related to ours. They considered finite-state compact-control SSP games in which one-stage costs can take both positive and negative values, and the termination state need not be reachable for every initial state and every pair of policies of the two players. 
The term ``SSP games'' is, in fact, from~\cite{sspgame_pt99}, and it is based on the close connection of SSP games, at both analytical and computational level,  to single-player SSP problems, which are total cost or total reward Markov decision processes (MDP) with a termination state. (For references on SSP and total cost MDP, see e.g., Bertsekas and Tsitsiklis \cite{bet-ssp91, BET}, Feinberg~\cite{feinberg92}, and Puterman~\cite{puterman94}.)  
Patek and Bertsekas \cite{sspgame_pt99} established optimality results similar to those we aim to obtain, but under model conditions that are asymmetric in terms of the two players and bear a strong association with pursuit-evasion type of games.

As one of the contributions of this paper, we introduce a symmetric formulation of model conditions (Assumption~\ref{assum-app-sspgame}). It characterizes a much broader class of SSP games than considered in \cite{sspgame_pt99} (see Section~\ref{sec-ssp-remarks} for a detailed comparison), and it allows  
the theory of single-player SSP problems (Bertsekas and Tsitsiklis \cite{bet-ssp91}) to be more fully utilized in analyzing the compact-control SSP games. As a result, we show that the desired optimality properties, including the existence of a unique solution to the Bellman equation and the existence of  a pair of equilibrium policies that are stationary deterministic (Theorem~\ref{thm-opt}), as well as the convergence properties of value iteration and policy iteration (Theorem~\ref{thm-opt-vipi}), are retained.

We then consider finite state and control SSP games that satisfy the proposed model conditions, 
and we turn to the question of the convergence of the Q-learning algorithm for computing their value functions.
Q-learning was first introduced by Watkins~\cite{wat89} in the context of MDP and reinforcement learning, 
and its convergence was analyzed most comprehensively by Tsitsiklis~\cite{tsi94} as a special case of the convergence of asynchronous stochastic approximation algorithms.
For discounted stochastic games, Littman~\cite{littman96} studied Q-learning and analyzed its convergence (with a different argument than~\cite[Theorem 3]{tsi94}, which also implies the convergence of Q-learning in such games). 
For undiscounted SSP games whose Bellman equations admit a unique solution, convergence of Q-learning is known in two limited cases under strong assumptions: 
\begin{itemize}
\item[(i)] when the game always terminates regardless how the two players play, 
and 
\item[(ii)] when the iterates generated by Q-learning are bounded with probability one. 
\end{itemize}
In both cases, the convergence of Q-learning follows from the convergence theorems of Tsitsiklis~\cite{tsi94} for asynchronous stochastic approximation involving sup-norm contraction or monotone nonexpansive mappings: convergence in the first case is due to a contraction property (Patek and Bertsekas \cite[Lemma 4.1]{sspgame_pt99}), and convergence in the second case (under the boundedness condition) follows from arguments for monotone nonexpansive mappings \cite[Theorem 2]{tsi94}. (For more details, see Bertsekas and Tsitsiklis~\cite[Chapter 4 and Section 7.2]{BET}.) 
Another convergence result is also known when boundedness of Q-learning iterates is not assumed,  
based on the results of Abounadi, Bertsekas and Borkar~\cite{abb02}. However, in this case, additional conditions are required on the timing and frequency of component updates in Q-learning, which are more restrictive than the totally asynchronous computing framework of \cite{tsi94}. 

The main contribution of this paper is a boundedness proof for the Q-learning algorithm with totally asynchronous computation, for the broad class of SSP games satisfying our model conditions. We show that the Q-learning iterates are bounded with probability one  (Theorem~\ref{thm-boundql}), thereby furnishing the boundedness condition required in the convergence theorem of \cite{tsi94} and establishing completely the convergence of Q-learning (Theorem \ref{thm-convql}).
Our proof techniques are based on those constructed in Yu and Bertsekas \cite{yub_bd11} for analyzing boundedness of Q-learning in single-player SSP problems.

This paper is organized as follows. In Section~\ref{sec-game} we consider finite-state compact-control SSP games, and introduce our new model conditions and prove optimality results.
In Section~\ref{sec-finitessp} we describe finite state and control SSP games and the Q-learning algorithm. Finally, in Section~\ref{sec-bd} we present the boundedness analysis for Q-learning.

\section{A Finite-State Compact-Control SSP Game Model} \label{sec-game}

\subsection{Basic Definitions and Conditions} \label{sec-2.1}
We consider a finite-state two-player zero-sum total cost stochastic game with a termination state.  
Let $\Ste = S \cup \{0\}$ be the state space, where $S=\{1, \ldots, n\}$ and state $0$ is a cost-free termination (absorbing) state.
Two players participate in the game with opposite objectives, and their actions jointly influence the evolution of the states through time. In particular, at each state $i \in \St$, player I (player II, respectively) can apply a control from a set $\bar U(i)$ ($\bar V(i)$, respectively) of feasible controls,
where $\bar U(i)$ and $\bar V(i)$ are assumed to be compact sets in some complete separable metric space. 
If the two players apply a pair of controls $(\bar u, \bar v) \in \bar U(i) \times \bar V(i)$, an expected one-stage cost $c_i(\bar u, \bar v)$ is incurred to player I while player II receives the same amount as an expected one-stage reward, and
the system then transitions from state $i$ to state $j \in \Ste$ with probability $p_{ij}(\bar u, \bar v)$.
Here the one-stage costs (with respect to player I) can be positive or negative.
We assume that the transition probabilities and one-stage costs satisfy the following continuity/semi-continuity conditions:

\begin{assumption}[Continuity Condition] \label{assump-continuity}
For all states $i, j \in \St$, the transition probability $p_{ij}(\bar u, \bar v)$ is a continuous function on $\bar U(i) \times \bar V(i)$, and the one-stage cost $c_i(\bar u, \bar v)$ is lower semicontinuous in $\bar u$ for fixed $\bar v$ and upper semicontinuous in $\bar v$ for fixed $\bar u$.  
\end{assumption}

Starting from some state $i_0 \in \St$ at time $0$, the players play for an infinite number of stages, making control decisions based on the information of the current state and the history of the game, which includes the past states and past controls applied by each player, while the states evolve in a Markovian way as described above. We define the total costs for player I and the total rewards for player II as follows. 

Let $i_k$ denote the state and $(\bar u_k, \bar v_k)$ the controls taken by the two players at time $k$. Let $\Pi_1, \Pi_2$ denote the sets of all history-dependent randomized policies for player I and player II, respectively (each of such policies is a collection of Borel measurable transition probabilities from the space of histories to the respective player's control space).
If player I adopts policy $\pi_1 \in \Pi_1$ and player II $\pi_2 \in \Pi_2$, we define the total cost of player I  (total reward of player II) for the initial state $i_0=i$ by
$$ x_i(\pi_1, \pi_2) = \liminf_{t \to \infty}  \, \E_{\pi_1 \pi_2} \Big[ \sum_{k = 0}^t c_{i_k}(\bar u_k, \bar v_k) \, \big| \,  i_0 = i \, \Big],$$
where $\{ (i_k, \bar u_k, \bar v_k), k \geq 0\}$ is the random process of states and controls induced by the policy pair $(\pi_1, \pi_2)$, and $\E_{\pi_1 \pi_2}$ denotes expectation with respect to the probability distribution of the induced process. In vector notation we write $x(\pi_1, \pi_2)$ for the vector of total costs, $(x_1(\pi_1, \pi_2), \ldots, x_n(\pi_1, \pi_2))$.

The optimal total cost for player I and optimal total reward for player II, for each initial state $i \in S$, are defined  to be 
$$ \bar x^*_i = \inf_{\pi_1 \in \Pi_1} \sup_{\pi_2 \in \Pi_2} x_i(\pi_1, \pi_2), \qquad \quad  \underline{x}^*_i = \sup_{\pi_2 \in \Pi_2} \inf_{\pi_1 \in \Pi_1}  x_i(\pi_1, \pi_2),$$
respectively. An optimal policy for player I (player II) is then a policy which attains the optima for all states in the above minimization over $\Pi_1$ (maximization over $\Pi_2$). We call $\bar x^*_i, \underline{x}^*_i$ the \emph{upper and lower value} of the game for state $i$. If these values coincide for all states, we call the corresponding $x^*=(x^*_1, \ldots, x^*_n)$ where $x^*_i=\bar x^*_i = \underline{x}^*_i$, the \emph{value function} of the game.
We say that $(\pi_1^*, \pi_2^*) \in \Pi_1 \times \Pi_2$ is a  pair of \emph{equilibrium policies} if the following holds:
$$   x(\pi_1^*, \pi_2) \leq x(\pi_1^*, \pi_2^*) \leq x(\pi_1, \pi_2^*), \qquad  \ \forall \, \pi_1 \in \Pi_1,  \ \pi_2 \in \Pi_2.$$
In that case $x^* = x(\pi_1^*, \pi_2^*)$ is the value function of the game, and $\pi_1^*, \pi_2^*$ are optimal policies for the two players.

Consider the class of stationary deterministic policies of each player, which is defined for player I and player II by  
\begin{align*}
  \nDa & = \Big\{ \npla : \St \mapsto \cup_{i \in \St} \bar U(i)  \ \Big| \  \npla(i)  \in \bar U(i),  \ i \in \St \Big\}, \\
   \nDb & = \Big\{ \nplb : \St \mapsto \cup_{i \in \St} \bar V(i) \ \Big| \  \nplb(i)  \in \bar V(i),  \ i \in \St  \Big\},
\end{align*}
respectively. Each function $\mu \in \nDa$ corresponds to a policy that applies at time $k$ the control $\mu(i_k)$ for state $i_k$, and this policy will also be denoted by $\mu$. We use similar notation for the policies corresponding to $\nDb$.
We will shortly introduce model conditions that guarantee the existence of equilibrium policies within these policies.\footnote{Without loss of generality, we focus on stationary deterministic policies here instead of stationary randomized policies, because our results can be applied in compact-control problems after a reformulation that let $\bar U(i)$ and $\bar V(i)$ represent probability distributions over the actual control sets.}

With stationary policies in $\nDa,\nDb$, we define several dynamic programming operators on $\rn$ for the game, using compact matrix and vector notation. 
For a pair of policies $(\npla, \nplb) \in \nDa \times \nDb$,  let $T_{\npla\nplb}: \rn \to \rn$ be given by
\begin{equation}
   T_{\npla\nplb} \, x  = c(\npla, \nplb) + P(\npla, \nplb)  x, \qquad x \in \rn,
\end{equation}
where $c(\npla, \nplb)$ is the $n$-dimensional one-stage cost vector with components $c_i\big(\npla(i), \nplb(i) \big)$, 
and $P(\npla, \nplb)$ is the $n$-by-$n$ substochastic transition probability matrix with elements $[P(\npla, \nplb)]_{ij} = p_{ij}\big(\npla(i), \nplb(i) \big)$, $i, j \in \St$. 
Define $T_{\npla}: \rn \to \rn$ and $T_{\nplb}: \rn \to \rn$ by
\begin{align}  \label{eq-dpoperator1}
    T_\npla x & = \sup_{\nplb \in \nDb} \big\{ c(\npla, \nplb) + P(\npla, \nplb)  x \big\},   & \tilde T_{\nplb} x & =  \inf_{\npla \in \nDa}  \big\{ c(\npla, \nplb) + P(\npla, \nplb)  x \big\}.
\end{align}
In the right-hand sides above the optimization over $\nDa$ or $\nDb$ is component-wise.
\footnote{Here we use the matrix/vector notation to write $n$ optimization problems in one expression. This is valid because of the separable structure of these problems. For example, the problem of maximizing the $i$th component of $c(\npla, \nplb) + P(\npla, \nplb)  x$ over $\nDb$ is identical to $\sup_{\nplb(i) \in \bar V(i)}  \{ c_i\big(\npla(i), \nplb(i) \big) +  \sum_{j\in \St} p_{ij}\big(\npla(i), \nplb(i) \big) x_j \}$. 
In other words, the $i$th optimization problem depends only on the components of $\npla, \nplb$ for state $i$.}

Finally, we define $T: \rn \to \rn$ and $\tilde T: \rn \to \rn$ by
\begin{align} \label{eq-dpoperator2}
    T x & = \inf_{\npla \in \nDa} \sup_{\nplb \in \nDb} \big\{ c(\npla, \nplb) + P(\npla, \nplb) x \big\},   &  \tilde T x & = \sup_{\nplb \in \nDb} \inf_{\npla \in \nDa}  \big\{ c(\npla, \nplb) + P(\npla, \nplb) x \big\},
\end{align}
where, similar to the above, the optimization in the right-hand sides is component-wise.

A mapping $H$ is monotone if $H x \leq H y$ for $x \leq y$. Since $P(\npla, \nplb)$ is a nonnegative matrix, the above mappings are monotone by definition. They also satisfy, be definition,
\begin{align}
 \tilde T_\nplb x \leq T_{\npla\nplb} x \leq T_\npla x,  \qquad  \quad \forall \, \npla \in \nDa, \ \nplb \in \nDb, \ x \in \rn, \label{app-ineq-1}\\
  \tilde T_\nplb x \leq   \tilde T x \leq T x  \leq T_\npla x, \qquad   \quad \forall \, \npla \in \nDa, \ \nplb \in \nDb, \ x \in \rn. \label{app-ineq-2}
\end{align}
Furthermore, Assumption~\ref{assump-continuity} on the continuity of the state transition probabilities and the semicontinuity of the one-stage costs implies that every component of $T_\npla x$ is lower semicontinuous in $(x, \npla)$, every component of $\tilde T_\nplb x$ is upper semicontinuous in $(x, \nplb)$, and every component of $T_{\npla\nplb} x$ is lower semicontinuous in $(x, \npla)$ for fixed $\nplb$ and upper semicontinuous in $(x, \nplb)$ for fixed $\npla$. Since the control sets are compact, it then follows that under Assumption~\ref{assump-continuity}, the infimum and supremum in the definitions of the above mappings are all attained: for every $x$,
there exists $\npla$ such that $T x = T_\npla x$; for every $x$ and $\npla$, there exists $\nplb$ such that $T_{\npla} x = T_{\npla \nplb} x$; and similar relations hold for $\tilde T$ and $\tilde T_{\nplb}$.

We also need a regularity condition:
\begin{assumption}[Minimax Regularity Condition] \label{assump-regularity}
For all $x \in \rn$, we have $T x = \tilde T x$, i.e.,
$$ \inf_{\npla \in \nDa} \sup_{\nplb \in \nDb} T_{\npla\nplb} \, x  = \sup_{\nplb \in \nDb} \inf_{\npla \in \nDa} T_{\npla\nplb} \, x.$$
\end{assumption}
\noindent Assumption~\ref{assump-regularity} is known to hold for cases where the control sets $\bar U(i)$ and $\bar V(i)$ in the above mathematical model correspond to the sets of probability distributions over the actual control sets which are compact, under certain continuity/semi-continuity conditions that can be weaker than Assumption~\ref{assump-continuity}. (See, for instance, \cite[Theorem 5.1]{game_No85}; see also the minimax theorems of Fan~\cite{fan53} for various conditions under which the above assumption holds.) 
In particular,
Assumption~\ref{assump-regularity}, as well as Assumption~\ref{assump-continuity}, 
is satisfied by the finite-space total cost zero-sum games that we will consider later.   
Under this assumption, we refer to $T$ or $\tilde T$ as the \emph{dynamic programming operator} and the equation $x = Tx$ or $x = \tilde T x$ as the \emph{dynamic programming equation} for the SSP game.

\subsection{An SSP Game Model and its Optimality Properties} \label{sec-sspgame-model}

We now introduce a novel formulation of an SSP game model. We will show that it has favorable optimality properties, including the existence of a unique solution to the dynamic programming equation, the existence of a pair of stationary equilibrium policies, and convergence of value and policy iteration. 

We will put model assumptions on the cost/reward of certain policies depending on whether the termination state can be reached with probability $1$ (w.p.$1$, for short). We need the following definition, which uses terminologies from~\cite{sspgame_pt99}. 

\smallskip
\begin{definition}[Prolonging and Non-prolonging Policies]
We say a pair of policies $(\pi_1, \pi_2)$ is \emph{prolonging} , if under these policies of the two players, there is a positive probability that the termination state $0$ is never reached for some initial state.
Then, a \emph{non-prolonging} pair $(\pi_1, \pi_2)$ is one such that under these policies, the termination state is reached for any initial state w.p.$1$. 
\end{definition}

{\samepage
\begin{assumption}[SSP Game Model] \label{assum-app-sspgame}
\hfill \moveup\moveup
\begin{itemize}
\item[(i)] There exists a policy $\nbpla \in \nDa$ for player I such that for any policy $\nplb \in \nDb$, 
$x_i(\nbpla, \nplb) < + \infty$ for all states $i$. 
\moveup
\item[(ii)] There exists a policy $\nbplb \in \nDb$ for player II such that for any policy $\npla \in \nDa$, 
$x_i(\npla, \nbplb) > - \infty$ for all states $i$.  
\moveup
\item[(iii)] For any pair of  policies $(\npla, \nplb) \in \nDa \times \nDb$ that is prolonging, $x_i(\npla, \nplb) = + \infty$ or $- \infty$ for at least one initial state $i$.\moveup
\end{itemize}
\end{assumption}
}
\smallskip

Assumption~\ref{assum-app-sspgame} has a symmetric form for the two players.
\footnote{ 
Because of the use of liminf, 
the definition of the total cost function $x(\pi_1, \pi_2)$ for a pair of general  policies $(\pi_1, \pi_2)$ is asymmetric for the two players. However, for a pair of stationary policies $(\npla, \nplb) \in \nDa \times \nDb$, it can be shown that under Assumption~\ref{assum-app-sspgame}, the limit of the finite-stage costs (or rewards) always exists (it may be finite, $+ \infty$ or $- \infty$). Because of this, the model assumption we introduce is indeed fully symmetric in terms of the two players.}
It is much broader than the asymmetric SSP model formulation in the earlier work \cite{sspgame_pt99}, as we will explain in Section~\ref{sec-ssp-remarks}.
Assumption~\ref{assum-app-sspgame}(i)-(ii) says that each player has at least one stationary policy to safeguard against infinite loss.
Assumption~\ref{assum-app-sspgame}(iii) says that a prolonging policy pair $(\npla, \nplb)$ will be against the interest of some player.
It also implies that the pair $(\nbpla, \nbplb)$ of policies described in Assumption~\ref{assum-app-sspgame}(i)-(ii) cannot be prolonging.

To derive further implications of Assumption~\ref{assum-app-sspgame}, we consider the decision problem for one player when the other player plays a fixed stationary policy. In that case, the problem of optimizing the total cost or reward for one player is a total cost or reward MDP with a cost-free termination state. For these finite-state compact-control MDP, strong optimality properties are known under certain assumptions on the total cost/reward structure (Bertsekas and Tsitsiklis~\cite{bet-ssp91}). Like \cite{sspgame_pt99} on SSP games, our SSP game model in Assumption~\ref{assum-app-sspgame} is also motivated by these analytical results for single-player problems.

More specifically, let us consider a single-player problem which, when viewed as a two-player game by assuming there is a second dummy player who has singleton control sets, satisfies the model description in Section~\ref{sec-2.1}, including the continuity conditions in Assumption~\ref{assump-continuity}. We will refer to such a problem as a single-player SSP problem, whether it is to minimize total costs or to maximize total rewards. For a single-player SSP, we have from~\cite{bet-ssp91} the following notion of proper policies and a model condition that uses this notion and leads to desirable optimality properties.

\smallskip
\begin{definition}[Proper and Improper Policies in Single-Player SSP] \label{def-proper-ssp}
In a single-player SSP problem, a policy is said to be \emph{proper}  if under that policy, the termination state is reached w.p.$1$ for any initial state; the policy is said to be \emph{improper}, otherwise.
\end{definition}
\smallskip

\noindent The results of~\cite{bet-ssp91} show that if a single-player SSP problem satisfies the following assumption,\\ 
\hspace*{0.25cm} \parbox{0.95\textwidth}{ \vspace*{0.3cm}
 {\bf SSP Model Assumption:}  \it In the class of stationary deterministic policies, 
 there exists a proper policy, and every improper policy incurs cost $+\infty$ for at least one initial state.  \vspace*{0.25cm}}\\
then 
the optimal total cost function is finite,
and it is the unique solution of the dynamic programming equation. Moreover, value iteration converges starting from any initial value. 

Based on these results for single-player SSP, let us introduce a notion of well-behaved policies for each player in SSP games.
Let us call a policy of player I or player II \emph{essentially proper} if, when the player plays that policy, the resulting (total cost or total reward) single-player SSP problem for the other player satisfies the SSP Model Assumption. In other words: 

\smallskip
\begin{definition}[Essentially Proper Policies] \label{def-essproper-sspgame} \hfill
\moveup \moveup
\begin{itemize}
\item[(a)] $\npla \in \nDa$ is \emph{essentially proper} if there exists a policy $\nplb \in \nDb$ such that $(\npla, \nplb)$ is non-prolonging, 
and moreover, for every policy $\nplb \in \nDb$ with $(\npla, \nplb)$ being prolonging, $x_i(\npla, \nplb) = - \infty$ for at least one initial state $i$;\moveup
\item[(b)] $\nplb \in \nDb$ is \emph{essentially proper} if there exists a policy $\npla \in \nDa$ such that $(\npla, \nplb)$ is non-prolonging, 
and moreover, for every policy $\npla \in \nDa$ with $(\npla, \nplb)$ being prolonging, $x_i(\npla, \nplb) = + \infty$ for at least one initial state $i$.\moveup
\end{itemize}
\end{definition}
\smallskip

If player I plays an essentially proper policy $\npla$, the reward-maximization problem player II faces is a single-player total-reward SSP with its dynamic programming operator given by $\tilde T_\npla$ [cf.\ Eq.~(\ref{eq-dpoperator1})]. Similarly, if player II plays an essentially proper policy $\nplb$, then player I has a single-player total-cost SSP problem with its dynamic programming operator given by $\tilde T_\nplb$ [cf.\ Eq.~(\ref{eq-dpoperator1})]. Hence by \cite{bet-ssp91} the essentially proper policies we just defined have the following property.

\smallskip
\begin{lemma} \label{lma-app-properext0}
Let $H = T_\npla$ or $\tilde T_\nplb$, where $\npla \in \nDa$ or $\nplb \in \nDb$ is essentially proper.
Then the equation $x  = H x$ has a unique solution $\bar x$,  and $ \lim_{t \to \infty} H^t x = \bar x$ for all $x \in \rn$.
\end{lemma}
\smallskip

For any pair of essentially proper policies of the two players, we have the following fact:

\smallskip
\begin{lemma}\label{lma-app-properext1}
Let $\npla \in \nDa$ and $\nplb \in \nDb$ be essentially proper. 
Then,\moveup\moveup 
\begin{itemize}
\item[(i)] $(\npla, \nplb)$ is non-prolonging; and 
\item[(ii)] $\bar x(\npla) \geq \tilde x(\nplb)$, where $\bar x(\npla), \tilde x(\nplb)$ are the unique solution of $x = T_\npla x$ and $x = \tilde T_{\nplb}x$, respectively. \moveup
\end{itemize}
\end{lemma}
\begin{proof}
To prove (i), first we note that although for a policy $\npla$ to be essentially proper, Definition~\ref{def-essproper-sspgame}(a)  does not exclude that $x_i(\npla, \nplb) = + \infty$ for some state $i$ and a prolonging policy pair $(\npla, \nplb)$,  this cannot happen.
Otherwise, we can derive a contradiction by constructing a policy $\nplb' \in \nDb$ for player II such that $(\npla, \nplb')$ is prolonging but $x_i(\npla, \nplb') > - \infty$ for all states $i$. 
This policy $\nplb'$ can be chosen as follows. Suppose $x_i(\npla, \nplb) = + \infty$ for some state $i$. Then, for the Markov chain induced by $(\npla, \nplb)$,  by \cite[Theorem 9.4.1, p.\ 472]{puterman94}, there exists a recurrent class $E$ such that the average cost on $E$ is strictly greater than $0$. Let $\nplb'$ be identical to $\nplb$ for states in $E$ and identical to a policy $\nbplb$ for the rest of the states, where $\nbplb$ is such that $(\npla, \nbplb)$ is non-prolonging and the existence of $\nbplb$ is ensured by the definition of $\npla$ as an essentially proper policy. The Markov chain induced by $(\npla, \nplb')$ has two recurrent classes, $E$ and $\{0\}$, so $(\npla, \nplb')$ is prolonging, and moreover, the average cost on $E$, [which is equal to the average cost on $E$ under $(\npla, \nplb)$], is strictly greater than $0$. Then, by \cite[Theorem 9.4.1, p.\ 472]{puterman94}, $x_i(\npla, \nplb') = + \infty$ for all $i \in E$, while for the rest of the states $i$, we have that either $x_i(\npla, \nplb') = + \infty$ or $x_i(\npla, \nplb')$ is finite. Hence, $x_i(\npla, \nplb') > - \infty$ for all $i$, and since the pair of policies $(\npla, \nplb')$ is prolonging, this contradicts the definition of $\npla$ being essentially proper.
Similarly, in Definition~\ref{def-essproper-sspgame}(b), it cannot happen that  $x_i(\npla, \nplb) = - \infty$ for some state $i$ and a prolonging policy pair $(\npla, \nplb)$ when $\nplb$ is essentially proper. 

On the other hand, when $\npla$ and $\nplb$ are essentially proper, Definition~\ref{def-essproper-sspgame} dictates that if $(\npla, \nplb)$ were prolonging, there must exist some states $i, j$ with $x_i(\npla, \nplb) = -\infty$ and $x_j(\npla, \nplb) = +\infty$, which is impossible as we just argued. Therefore, the pair $(\npla, \nplb)$ must be non-prolonging.

We now prove (ii). Since $\npla$ and $\nplb$ are essentially proper, by Lemma~\ref{lma-app-properext0}, the equations 
$x = T_\npla x$ and $x = \tilde T_{\nplb}x$ have a unique solution.
Denote $\bar x = \bar x(\npla), \tilde x = \tilde x(\nplb)$.  Since $ \bar x = T_\npla \bar x \geq \tilde T_{\nplb} \bar x$ [cf.\ Eq.~(\ref{app-ineq-1})] and $\tilde T_{\nplb}$ is monotone, we have that for all $t$, $\bar x \geq {\tilde T}^t_{\nplb} \bar x$. By Lemma~\ref{lma-app-properext0}, $\lim_{t \to \infty} {\tilde T}^t_{\nplb} \bar x  = \tilde x$. Therefore $\bar x \geq \tilde x$.
\end{proof}
\smallskip

In terms of essentially proper policies, Assumption~\ref{assum-app-sspgame} has an important implication given below.

\smallskip
{\samepage
\begin{lemma} \label{lma-app-properext}
Suppose Assumption~\ref{assum-app-sspgame} holds. Let $\npla \in \nDa$ and $\nplb \in \nDb$. Then we have:
\moveup\moveup
\begin{itemize}
\item[(i)] If there exists $x \in \rn$ such that $x \geq T_\npla x$, then $\npla$ is essentially proper. \moveup\moveup 
\item[(ii)] If there exists $x \in \rn$ such that $x \leq \tilde T_\nplb x$, then $\nplb$ is essentially proper. \moveup
\end{itemize}
Moreover, every player has at least one essentially proper stationary policy---$\nbpla$ for player I and $\nbplb$ for player II where $\nbpla$ and $\nbplb$ are as in Assumption~\ref{assum-app-sspgame}(i)-(ii).
\end{lemma}
}

\begin{proof}
We prove (i); the proof for (ii) is entirely symmetric. 
For any policy $\nplb \in \nDb$, since $x \geq T_\npla x \geq T_{\npla\nplb} x$ [cf.\ Eq.~(\ref{app-ineq-1})], 
by the monotonicity of $T_{\npla\nplb}$ and the definition of $x_i(\npla, \nplb)$, we have that $x_i(\npla, \nplb) < + \infty$ for any state $i$. 
Then for the policy $\nbplb$ of player II in Assumption~\ref{assum-app-sspgame}(ii), the pair $(\npla, \nbplb)$ must be non-prolonging by Assumption~\ref{assum-app-sspgame}(iii), and also by Assumption~\ref{assum-app-sspgame}(iii),
for every $\nplb \in \nDb$ such that $(\npla, \nplb)$ is prolonging, $x_i(\npla, \nplb) = - \infty$ for at least one state $i$. By Definition~\ref{def-essproper-sspgame}(a), this shows that $\npla$ is essentially proper. 

We now prove the last statement of the lemma. Consider the policies $\nbpla$ and $\nbplb$ in Assumption~\ref{assum-app-sspgame}(i) and (ii). As discussed immediately after that assumption, the pair $(\nbpla, \nbplb)$ is non-prolonging. Moreover, Assumption~\ref{assum-app-sspgame}(i) and (iii) together imply that for any policy $\nplb \in \nDb$ such that $(\nbpla, \nplb)$ is prolonging, we must have $x_i(\nbpla, \nplb) = - \infty$ for some initial state $i$. Hence $\nbpla$ is essentially proper for player I by Definition~\ref{def-essproper-sspgame}(a). Similarly, the policy $\nbplb$ is essentially proper for player I by Assumption~\ref{assum-app-sspgame}(ii)-(iii) and Definition~\ref{def-essproper-sspgame}(b).
\end{proof}
\smallskip

We are now ready to establish the optimality properties for the proposed SSP game model. Some of the proof steps below appear similar to those in \cite{sspgame_pt99}.

\smallskip
\begin{theorem}[Existence of Value and Equilibrium Policies] \label{thm-opt}
Under Assumptions~\ref{assump-continuity}-\ref{assum-app-sspgame}, the game has a finite value function $x^*$, which is the unique solution of the dynamic programming equation $x = T x$. Furthermore, any $\npla^* \in \nDa$, $\nplb^* \in \nDb$ such that $x^* = T_{\npla^*} x^* = \tilde T_{\nplb^*} x^*$ are essentially proper (hence $(\npla^*, \nplb^*)$ is non-prolonging). Such policies exist, and they form a pair of equilibrium policies for the game and are optimal for each player.
\end{theorem}

\begin{proof}  
We show first that $T$ can have at most one fixed point. 
Suppose that both $x$ and $x'$ satisfy $x = T x$ and $x' = T x'$. 
Under Assumption~\ref{assump-continuity}, there exist stationary deterministic policies $\npla$ and $\npla'$ such that $T_\npla x = T x$ and $T_{\npla'} x' = T x'$. By Lemma~\ref{lma-app-properext}, both $\npla$ and $\npla'$ are essentially proper.
Since $x = T x \leq T_{\npla'} x$ [cf.\ Eq.~(\ref{app-ineq-2})], by the monotonicity of $T_{\npla'}$, we have that for all $t$, $x \leq T_{\npla'}^t x$. On the other hand, since $x' = T_{\npla'} x'$ and $\npla'$ is essentially proper, we have by Lemma~\ref{lma-app-properext0} that $\{ T_{\npla'}^t x\}$ converges to $x'$. Therefore $x \leq x'$. A symmetric argument yields $x' \leq x$, and hence $x = x'$.

We now show that $T$ has a fixed point. Let $\nbpla \in \nDa$ and $\nbplb \in \nDb$ be  essentially proper policies, which exist under Assumption~\ref{assum-app-sspgame} (Lemma~\ref{lma-app-properext}).
By Lemma~\ref{lma-app-properext0}, there exist a unique $\bar x$ such that $\bar x = T_{\nbpla} \bar x$, and a unique $\tilde x$ such that $\tilde x = \tilde T_{\nbplb} \tilde x$.
By Lemma~\ref{lma-app-properext1}(ii), $\bar x \geq \tilde x$.
Since $T_{\nbpla} \bar x \geq T \bar x$ and $ \tilde T_{\nbplb} \tilde x  \leq  T \tilde x$ [cf.\ Eq.~(\ref{app-ineq-2})],
we  also have $\bar x \geq T \bar x$ and  $\tilde x  \leq  T \tilde x$.
Using the monotonicity of $T$ and the fact that $\bar x \geq \tilde x$, it follows that
$\{T^t \bar x\}$ is a non-increasing sequence bounded below by $\tilde x$ and hence converges to some $x^*$.
Since $T^{t+1} \bar x = T (T^t \bar x)$,  by the continuity of $T$, $x^*$ must satisfy $x^* = T x^*$. Thus $x^*$ is a fixed point of $T$ and hence the unique fixed point of $T$.

Now let $(\npla^*, \nplb^*)$ be stationary deterministic policies such that $T_{\npla^*} x^* = T x^*$ and $\tilde T_{\nplb^*} x^* = \tilde T x^*$; they exist under Assumption~\ref{assump-continuity}.
We have $x^* =  T_{\npla^*} x^* =  \tilde T_{\nplb^*} x^*$ because $T x^* = \tilde T x^*$ under Assumption~\ref{assump-regularity}. 
Lemma~\ref{lma-app-properext}(i) and (ii) then imply that $\npla^*$ and $\nplb^*$ are essentially proper, so 
by the result of \cite{bet-ssp91}, $x^*$ is the optimal total reward function 
(optimal total cost function, respectively) of the single-player SSP problem for player II (player I, respectively) when player I takes policy $\npla^*$ (player II takes policy $\nplb^*$, respectively).
This optimality of $x^*$ translates to
\footnote{More precisely, the argument for $x(\npla^*, \pi_2) \leq x^*$, $\pi_2 \in \Pi_2$, is the following. For the total reward SSP problem resulting from player I taking policy $\npla^*$,
consider the corresponding total cost problem with one-stage costs being $-c_i(\bar u, \bar v)$. Then, by~\cite{bet-ssp91}, $-x^*$ is the optimal total cost function, and hence, for every state $i$ and $\pi_2 \in \Pi_2$,
\begin{align*}
-x^*_i & \leq \liminf_{t \to \infty} \E_{\npla^*\pi_2} \Big[ - \sum_{k=0}^t c_{i_k} (\bar u_k, \bar v_k) \mid i_0 = i \Big]  \\
  & = -  \limsup_{t \to \infty} \E_{\npla^*\pi_2} \Big[  \sum_{k=0}^t c_{i_k}(\bar u_k, \bar v_k) \mid i_0 = i \Big]  \leq - \liminf_{t \to \infty} \E_{\npla^*\pi_2} \Big[  \sum_{k=0}^t c_{i_k}(\bar  u_k, \bar v_k) \mid i_0 = i \Big] =  - x_i(\npla^*, \pi_2),
\end{align*}  
which is  $x^* \geq x(\npla^*, \pi_2)$.
}
\begin{equation} \label{prf-thm2.1-a}
  x(\npla^*, \pi_2) \leq x^* \leq x(\pi_1, \nplb^*), \qquad \forall \, \pi_1 \in \Pi_1, \pi_2 \in \Pi_2.
\end{equation}  

We now prove $x^* = x(\npla^*, \nplb^*)$. Since $\npla^*$ and $\nplb^*$ are essentially proper, by Lemma~\ref{lma-app-properext1}(i), $(\npla^*, \nplb^*)$ is non-prolonging. Applying the result of \cite{bet-ssp91} to the process induced by the non-prolonging pair $(\npla^*, \nplb^*)$, which can be viewed as an uncontrolled SSP with a single (dummy) proper policy, we obtain that the total cost function under $(\npla^*, \nplb^*)$ is the unique solution of the dynamic programming equation $ x = T_{\npla^* \nplb^*}  x$. On the other hand,
we have $T_{\npla^* \nplb^*} x^* = x^*$ because $x^* = \tilde T_{\nplb^*} x^* \leq T_{\npla^* \nplb^*} x^* \leq T_{\npla^*} x^* = x^*$ [cf.\ Eq.~(\ref{app-ineq-1})].
Therefore, $x^* = x(\npla^*, \nplb^*)$. Combining this with Eq.~(\ref{prf-thm2.1-a}), we then have that $(\npla^*, \nplb^*)$ is a pair of equilibrium (and optimal) policies for the two players and $x^*$ is the value function of the game.
\end{proof}

Next we consider value and policy iteration. Recall a well-known fact: if a monotone operator $H: \rn \to \rn$ is nonexpansive with respect to the sup-norm (i.e., $\|H x - H y \|_\infty \leq \| x - y \|_\infty$) and has a unique fixed point $\bar x$, then fixed point iterations $H^k x$ converge to $\bar x$ for any initial $x$ (see e.g., \cite[Lemma 2.1]{yb_ssp11} for a proof). The monotone mapping $T$ is nonexpansive with respect to the sup-norm, and under Assumptions~\ref{assump-continuity}-\ref{assum-app-sspgame}, it has a unique fixed point by Theorem~\ref{thm-opt}. Therefore, the iterates $\{ x_t\}$ generated by value iteration, $x_{t+1} = T x_t$, converge to $x^*$ for any initial $x_0 \in \rn$. 

Policy iteration for each player starting with an essentially proper policy also converges under Assumptions~\ref{assump-continuity}-\ref{assum-app-sspgame}. This is shown below. Since our SSP game model is symmetric for the two players, it suffices to discuss the case of player I. In policy iteration,
starting from a policy $\npla_0 \in \nDa$ that is essentially proper, we define recursively $x_t \in \rn$ and policy $\npla_{t+1} \in \nDa$ by
\begin{equation} \label{eq-app-polite}
    x_t = T_{\npla_t} x_t, \qquad  T_{\npla_{t+1}} x_t = T x_t, \qquad t \geq 0.
 \end{equation}   
By induction, in the above $x_t$ is well-defined (Lemma~\ref{lma-app-properext0}), $\npla_{t+1}$ is well-defined under Assumption~\ref{assump-continuity}, and 
since $x_t \geq T_{\npla_{t+1}} x_t$, all $\npla_{t+1}$ thus generated are essentially proper (Lemma~\ref{lma-app-properext}). 
It can also be seen that $\{ x_t\}$ is a non-increasing sequence (using the fact that $T_{\npla_{t+1}} x_t \leq x_t$).
We summarize these results in the theorem below.

\smallskip
\begin{theorem}[Convergence of Value and Policy Iteration] \label{thm-opt-vipi} 
Under Assumptions~\ref{assump-continuity}-\ref{assum-app-sspgame}, with $x^*$ being the value function of the game, the following holds: \moveup\moveup
\begin{itemize}
\item[(i)] Convergence of value iteration: For any $x \in \rn$, $\lim_{t \to \infty} T^t x = x^*$.
\item[(ii)] Convergence of policy iteration: Let $\{ x_t\}$ and $\{ \npla_t\}$ be defined by Eq.~(\ref{eq-app-polite}) with $\npla_0$ being essentially proper for player I.
Then all $\npla_t$ are essentially proper. Furthermore, $\lim_{t \to \infty} x_t = x^*$, and any cluster point $\npla_\infty$ of $\{\npla_t\}$ is essentially proper and optimal for player I.\moveup 
\end{itemize}
\end{theorem}

\begin{proof}
We prove the last statement in (ii); the other statements are already proved in the preceding discussion.
We have the relation
$$ x_t \geq T x_t = T_{\npla_{t+1}} x_t   \geq x_{t+1}.$$
(To see this, note that since $T_{\npla_t} x_t \geq T x_t$ [cf.\ Eq.~(\ref{app-ineq-2})], we have $x_t \geq T x_t = T_{\npla_{t+1}} x_t$. Using the monotonicity of $T_{\npla_{t+1}}$, the fact that $\npla_{t+1}$ is essentially proper, and Lemma~\ref{lma-app-properext0}, we then obtain $T_{\npla_{t+1}} x_t \geq x_{t+1}$.)
Hence the sequence $\{ x_t\}$ is non-increasing. 
Since all $\npla_t$ are essentially proper, 
by Lemma~\ref{lma-app-properext1}(ii), $\{x_t\}$ is bounded below by $\tilde x \in \rn$,  the unique fixed point of $\tilde T_{\nbplb}$, where $\nbplb$ is any policy of player II that is essentially proper. (By Lemma~\ref{lma-app-properext} such a policy $\nbplb$ exists under Assumption~\ref{assum-app-sspgame}.) 
Therefore, $\{x_t\}$ converges to some $x_\infty \in \rn$.
Using the relation $x_t \geq T x_t \geq x_{t+1}$ and the continuity of $T$, we obtain that $x_\infty \geq T x_\infty \geq x_\infty$, i.e., $x_\infty = T x_\infty$. Since $x^*$ is the unique fixed point of $T$ (Theorem~\ref{thm-opt}), we have $x_\infty = x^*$. 

Let $\npla_\infty$ be a cluster point of $\{\npla_t\}$. Since every component of $T_{\npla} x$ is a lower semicontinuous function of $(\npla, x)$ under Assumption~\ref{assump-continuity}, we obtain from the relation $x_t \geq T_{\npla_{t+1}} x_t $ and the convergence of $\{x_t\}$ to $x^*$ that $x^* \geq T_{\npla_\infty} x^*$. By Lemma~\ref{lma-app-properext}(i), this implies that $\npla_\infty$ is essentially proper.
We also have, by Eq.~(\ref{app-ineq-2}), that $T_{\npla_\infty} x^* \geq T x^* = x^*$. Hence $T_{\npla_\infty} x^* = T x^*$ and by Theorem~\ref{thm-opt}, $\npla_\infty$ is an optimal policy for player I.
\end{proof}
\smallskip

\subsection{Further Remarks} \label{sec-ssp-remarks}

The results we presented in this section bear close relations to those given in the earlier work \cite{sspgame_pt99} on SSP games.
In what follows we make a detailed comparison of our model assumption with the formulation in \cite{sspgame_pt99}, and we also discuss the scope and limitation of our model through a well-known example.

Patek and Bertsekas \cite{sspgame_pt99} formulated an SSP game model and derived optimality results similar to ours.
The model conditions of \cite[Assumption SSP]{sspgame_pt99} are stated in terms of deterministic Markov policies $\pi_1, \pi_2$ (instead of stationary policies) of the two players: 
 \moveup\moveup
\begin{itemize}
\item[(i)] There exists a policy $\pi_1$ of player I such that for all policies of player II, the termination state is reached w.p.$1$ for all initial states.\moveup
\item[(ii)] For every pair of policies $(\pi_1, \pi_2)$ that is prolonging, the expected total cost of player I is infinite for at least one initial state $i$, i.e., $x_i(\pi_1, \pi_2) = + \infty$.\moveup\moveup
\end{itemize}
Instead of the essentially proper policies as we have defined, the well-behaved policies in their framework, which they call proper policies, are the ones for player I under which the game terminates no matter how player II plays. Under the above assumptions and continuity and regularity conditions, they obtained optimality results similar to Theorems~\ref{thm-opt} and \ref{thm-opt-vipi} for their model. Their results ensure that player I has an optimal stationary proper policy and policy iteration for player I converges when starting from a proper policy. By contrast, our model formulation is centered on essentially proper policies, under which the game need not terminate for all policies of the other player, and our results ensure the existence of optimal stationary policies for both players within the class of essentially proper policies, as well as the convergence of policy iteration starting with such a policy.

Let us discuss more about the above model assumptions (i)-(ii) considered by \cite{sspgame_pt99}.
Because of the non-stationarity of policy $\pi_1$, it is not immediate to see what implication assumption~(i) has on the structure of the game. However, based on the results and analyses of~\cite{sspgame_pt99}, when assumption (ii) and other continuity/regularity conditions are in force, 
assumption (i) is equivalent to:
\begin{itemize}
\item[(i')] There exists a policy $\nbpla \in \nDa$ of player I such that for  all  $\nplb \in \nDb$ of player II, $(\nbpla, \nplb)$ is non-prolonging.\moveup\moveup
\end{itemize}
Assumptions (i') and (ii) clearly imply our Assumption~\ref{assum-app-sspgame}, so our SSP game model covers a larger class of games.
Assumptions (i)-(ii) or (i')-(ii) are also asymmetric for the two players, whereas Assumption~\ref{assum-app-sspgame} has a symmetric form.

To see why Assumption~\ref{assum-app-sspgame} characterizes a much broader class of games than the model conditions (i)-(ii) of \cite{sspgame_pt99} do, we note two main restrictions in the latter conditions.
First, assumption~(i) overly favors player I by requiring that player I can terminate the game however player II plays, whereas assumption (ii) overly favors player II by requiring that a non-terminating game is always to the disadvantage of player I. Although these conditions seem natural for those applications in which player I is a ``pursuer'' and player II an ``evader,'' and the game is over when the pursuer achieves the goal of catching the evader (see \cite[Sec.~5]{sspgame_pt99}), they are restrictive for games not of the pursuit-evasion type.
Second, by imposing the condition in assumption (ii) on \emph{every} pair of prolonging policies, assumptions (i)-(ii) effectively require 
that if the two players play only stationary policies, then against any given strategy of player II, player~I will not be able to obtain strictly negative average cost ($-\infty$ total cost), for any initial state. Consider the implication of this for a finite state and control game, for example. (In a finite state and control game, $\bar U(i)$ and $\bar V(i)$ correspond to the distributions over the finite control sets at state $i$ under randomized stationary policies, and $\pi_1$ and $\pi_2$ correspond to randomized Markov policies.) Then the requirement imposed by assumptions (i)-(ii) just mentioned entails that against any given \emph{deterministic} stationary policy of player II, player I cannot find a stationary policy to obtain an infinite amount of return by prolonging the game.
This is a serious restriction in the model formulation of~\cite{sspgame_pt99}.

To end this section, we discuss some examples of total cost zero-sum games that are excluded by our model assumptions. 
The following simple finite state and control game, due to Everett~\cite{game_everett57}, has no optimal policy for player II but has a value~\cite{KuS81}. The state space is $\Ste = \{0,1\}$. At state $1$, there are two controls $\{1,2\}$ for each player, and when player I applies control $u$ and player II control $v$, the system transitions to state $0$ with cost $1$ if $u=v$, transitions to state $0$ with cost $0$ if $u=1, v = 2$, and transitions to state $1$ with cost $0$ if $u=2, v = 1$. (Here $\bar U(1), \bar V(1)$ are given by the set of probability distributions on $\{1, 2\}$.) The value of the game for state $1$ is $1$. This example violates Assumption~\ref{assum-app-sspgame}(iii) because the pair of policies with player I applying control $2$ and player II control $1$ at state $1$, is prolonging but incurs zero total cost. 

It is worth to mention that in this example, although not every player has an optimal policy, the dynamic programming equation $x = Tx$ of the game does have a unique solution, which is the value of the game~\cite{KuS81}.  It is also easy to construct examples where the game has a value, both players have stationary optimal policies, and the dynamic programming equation has a unique solution, but there exists a prolonging pair of stationary optimal policies (with zero average cost) so that the game will be excluded by our model assumption. Here is the simplest such example: let the state and control spaces be as in the preceding example; let all one-stage costs be zero; and at state $1$, let the system transit to state $0$ if either player applies control $1$, and let the system stay at state $1$ otherwise. In this game, all policies are optimal, the value of the game is zero, and it is the unique solution of the dynamic programming equation. This illustrates that not all games with nice optimality properties are included in the class of games satisfying our model assumptions.

\section{Q-Learning for Finite-Space SSP Games} \label{sec-finitessp}  

Starting with this section, we will focus on total cost zero-sum games with a finite state and control space and analyze the convergence of a model-free, stochastic approximation-based algorithm, Q-learning, for solving these games. In this section, first, the SSP game model introduced in Section~\ref{sec-game} will be specialized to the finite-space game context, to provide a finite-space SSP game model that has desirable optimality properties for applying the Q-learning algorithm. The Q-learning algorithm will then be introduced, along with the convergence results we have obtained. The major proofs for these results will be given in the next section.

\subsection{Finite-Space SSP Games}

Consider a finite state and control two-player zero-sum game. 
The state space is $\Ste = \St \cup \{0\}$ as before, where $0$ is the cost-free termination state.
At state  $i \in \St$, each player has a finite set of feasible controls, denoted by $U(i)$, $V(i)$ for player I,  player II, respectively.
The rules of the game and the objectives of the two players are as described in Section~\ref{sec-2.1}. However, with apologies to the readers, we will use some different notation to make it conforming to standard notation in the Q-learning literature. 
In particular,
for each pair of controls 
$(u,v) \in U(i) \times V(i)$, let $p_{ij}(u,v)$ be the probability of transition from state $i$ to $j \in \Ste$, 
let $\hat g(i, u, v, j)$ be the corresponding transition cost,
\footnote{More generally, the transition cost can also depend on some additional stochastic disturbance $\omega$ and take the form $\hat g(i,u,v, j, \omega)$. Our analysis of Q-learning applies to such type of random transition costs provided that they have bounded variance, but for notational simplicity, we do not introduce them in the paper.}
and let $g(i, u,v) = \sum_{j \in \St} p_{ij}(u,v) \hat g(i, u, v, j)$ denote the expected one-stage cost at state $i$ with controls $(u,v)$.
At each time $t$, every player may use the information of the current state $i_t$ and the history of the game, including the past states $\{i_k, k < t\}$ and past controls $\{u_k, v_k, k < t\}$ of both players, to decide which control to apply.
When player I adopts policy $\pi_1$ and player II $\pi_2$, we write the total cost of player I starting from initial state $i$ as $J(i; \pi_1, \pi_2)$, i.e.,
$$ J(i; \pi_1, \pi_2) = \liminf_{t \to \infty}  \, \E_{\pi_1 \pi_2} \Big[ \sum_{k = 0}^t \hat g(i_k, u_k, v_k, i_{k+1}) \, \big| \,  i_0 = i \, \Big].$$
(Since we will deal with asynchronous iterative algorithms in this section, we find the notation $J(i; \pi_1, \pi_2)$ more convenient than the notation $x_i(\pi_1, \pi_2)$ of Section~\ref{sec-game}, thus reserving subscripts for iteration indices.)
As before, for every state $i$, the two players' goals are: 
$$ \text{player I:} \quad  \mathop{\text{minimize}}_{\pi_1 \in \Pi_1}  \sup_{\pi_2 \in \Pi_2} J(i; \pi_1, \pi_2), \qquad \ \ \text{player II:} \quad  \mathop{\text{maximize}}_{\pi_2 \in \Pi_2} \inf_{\pi_1 \in \Pi_1} J(i; \pi_1, \pi_2).$$

Of particular importance are stationary randomized policies. For each state $i \in \St$, let $\bar U(i) = \P\big(U(i)\big)$ and $\bar V(i) = \P\big(V(i)\big)$ denote the set of probability distributions on $U(i)$ and $V(i)$, respectively, which are the randomized decision rules of the two players for state $i$.
A stationary randomized policy of a player takes the form,
\begin{align*}
  \text{for player I:} &  \quad  \pla = \{ \pla(\cdot \mid i) \mid i \in \St \}, \ \ \text{where} \ \pla(\cdot \mid i) \in \bar U(i), \\
   \text{for player II:} & \quad \plb = \{ \plb(\cdot \mid i) \mid i \in \St \}, \ \ \text{where} \ \plb(\cdot \mid i) \in \bar V(i).
\end{align*}  
With such a policy $\pla$ ($\plb$, resp.), at state $i$, player I (player II, resp.) takes control $u$ ($v$, resp.) with probability $\pla(u \mid i)$ ($\plb(v \mid i)$, resp.). We denote the set of stationary randomized policies of player I  and player II by $\Ra$ and $\Rb$, respectively.

We can relate the above finite-space game to a finite-state compact-control game considered in Section~\ref{sec-game}, 
where the compact control sets correspond to the sets of randomized decision rules of each player in the present context. 
In particular, in the framework of Section~\ref{sec-game}, consider the corresponding compact-control game where:\moveup\moveup
\begin{itemize}
\item[(a)] The compact control sets at state $i$  for the two players are given by the sets $\bar U(i), \bar V(i)$ defined above.
For a pair $(\da, \db) \in \bar U(i) \times \bar V(i)$, the probability of transition to state $j$ is given by 
$\sum_{u \in U(i)} \sum_{v \in V(i)} \da(u) \db(v) p_{ij}(u, v),$ whereas the expected one-stage cost is given by 
 $\sum_{u \in U(i)} \sum_{v \in V(i)} \da(u) \db(v) g(i, u, v).$
These transition probabilities and one-stage costs satisfy the continuity/semi-continuity conditions in Assumption~\ref{assump-continuity}.
\item[(b)] The sets $\nDa$ and $\nDb$ of stationary deterministic policies in the notation of Section~\ref{sec-game} correspond to the sets $\Ra$ and $\Rb$ of stationary randomized policies of player I and player II defined above, respectively. \moveup
\item[(c)] With the correspondences in (a)-(b), the regularity condition in Assumption~\ref{assump-regularity} is satisfied, and the dynamic programming equation, which we write as $J = T J$ here, is given by
\begin{equation} \label{eq-Bellman}
     J(i) = (T J)(i) : = \inf_{\da \in \bar U(i)} \sup_{\db \in \bar V(i)}  \sum_{u \in U(i)} \sum_{v \in V(i)} \da(u) \db(v) \Big( g(i, u, v) +  \sum_{j \in S} p_{ij}(u, v) J(j) \Big), \quad \forall \, i \in \St.
\end{equation}
The dynamic programming operator $\tilde T$ is given by exchanging the order of $\inf$ and $\sup$ in the above expression defining $T$. The dynamic programming operators $T_{\pla}$, $\tilde T_{\plb}$ for policies $\pla \in \Ra$ and $\plb \in \Rb$ are given by
\begin{align*}
 (T_{\pla} J)(i) & : = \sup_{\db \in \bar V(i)}  \sum_{u \in U(i)} \sum_{v \in V(i)} \pla(u\mid i) \, \db(v) \Big( g(i, u, v) +  \sum_{j \in S} p_{ij}(u, v) J(j) \Big), \quad \forall \, i \in \St, \\
 (T_{\plb} J)(i) & : = \inf_{\da \in \bar U(i)}  \sum_{u \in U(i)} \sum_{v \in V(i)} \da(u) \, \plb(v\mid i) \Big( g(i, u, v) +  \sum_{j \in S} p_{ij}(u, v) J(j) \Big), \quad \forall \, i \in \St.
\end{align*} 
\end{itemize} 

The SSP game model given in Assumption~\ref{assum-app-sspgame} then translates to the following model condition on finite-space games:

\smallskip
\begin{assumption}[Finite-Space SSP Game Model] \label{assum-ssp} 
Assumption~\ref{assum-app-sspgame} holds for $\nDa = \Ra$ and $\nDb = \Rb$. That is, (i) player I (player II)  has a stationary randomized policy under which the player's total cost (reward) is less than $+ \infty$ (greater than $- \infty$) no matter what stationary randomized policy the other player takes; and (ii) under any prolonging pair of stationary randomized policies of the two players, there is some initial state for which either the total cost for player I is $+\infty$ or the total reward for player II  is $-\infty$.
\end{assumption}

Under Assumption~\ref{assum-ssp}, Theorems~\ref{thm-opt} and \ref{thm-opt-vipi} apply to finite-space games through their associated compact-control games just described. In particular, we obtain from Theorem~\ref{thm-opt}:
\footnote{In translating Theorem~\ref{thm-opt} into Proposition~\ref{prp-3.1},  
there is a small technical detail that we need to mention: the policy spaces $\Pi_1, \Pi_2$ in the finite-space game are not the policy spaces in the corresponding compact-control game. Let us denote the latter sets by $\bar \Pi_1, \bar \Pi_2$ for the two players respectively. In general a history-dependent policy in $\bar \Pi_1$ or $\bar \Pi_2$ does not necessarily lie in  $\Pi_1$ or $\Pi_2$. This is because the player in the finite-space game does not observe the randomized decision rules that the other player took in the past, and therefore cannot make control decisions based on that information, whereas the player in the corresponding compact-control game can use that information for control. However, Markov policies, in particular stationary policies, for either game are also policies for the other game. We use this fact together with a standard Markovian property in MDP to obtain the desired results for the finite-space game. For example, we can prove Eq.~(\ref{eq-equipol}) as follows. By a direct application of Theorem~\ref{thm-opt} to the compact-control game, there exist $(\pla^*, \plb^*) \in \Ra \times \Rb$ with
$$ J(i; \pla^*, \bar \pi_2) \leq J(i; \pla^*, \plb^*) \leq J(i; \bar \pi_1, \plb^*), \qquad  \forall \, \bar \pi_1 \in \bar \Pi_1,  \ \bar \pi_2 \in \bar \Pi_2, \ i \in \St.$$
To obtain Eq.~(\ref{eq-equipol}) from this inequality,
consider first the total cost $J(i; \pi_1, \plb^*)$ for any given state $i$ and policy $\pi_1 \in  \Pi_1$ in the finite-space game.
Because the state evolves in a Markovian way when player II plays the stationary policy $\plb^*$, 
one can construct a randomized Markov policy $\tilde \pi_1$ such that $J(i; \tilde \pi_1, \plb^*) = J(i;  \pi_1, \plb^*)$ (such construction is well-known in the MDP theory).
Since a randomized Markov policy of player I lies in the intersection $\Pi_1 \cap \bar \Pi_1$, we have $ J(i; \pla^*, \plb^*) \leq J(i; \tilde \pi_1, \plb^*)$ by the preceding inequality, and consequently, 
$J(i; \pla^*, \plb^*) \leq J(i;  \pi_1, \plb^*)$ for any $\pi_1 \in  \Pi_1$ and $i \in \St$. This proves the second half of the desired inequality~(\ref{eq-equipol}).
The other half of~(\ref{eq-equipol}) follows from the same argument applied to player II.
}

\smallskip
\begin{proposition}[Optimality Properties of Finite-Space SSP Games] \label{prp-3.1}
For a finite-space SSP game satisfying Assumption~\ref{assum-ssp}, there exist equilibrium policies $(\pla^*, \plb^*) \in \Ra \times \Rb$ for the two players, i.e.,
\begin{equation} \label{eq-equipol}
   J(i; \pla^*, \pi_2) \leq J(i; \pla^*, \plb^*) \leq J(i; \pi_1, \plb^*), \qquad  \forall \, \pi_1 \in \Pi_1,  \ \pi_2 \in \Pi_2, \ i \in \St.
\end{equation}   
The value function of the game, given by $J^*(\cdot) = J(\cdot; \pla^*, \plb^*),$
is the unique solution of the dynamic programming equation $J = TJ$ given by (\ref{eq-Bellman}).
Moreover, any $\pla^*\in \Ra$, $\plb^* \in \Rb$ such that $T_{\pla^*} J^* = T J^*$, $\tilde T_{\plb}^* J^* = T J^*$
are optimal policies of player I and player II, respectively, and they are essentially proper, with the pair $( \pla^*, \plb^*)$ forming a non-prolonging pair of equilibrium policies.
\end{proposition}

\smallskip
\begin{remark} \rm
In a sequential game, only one player can move at each time and whose turn to move depends on the current state (see e.g., \cite[Section 7.2]{BET}). Equivalently, at each state, one of the two players has a singleton control set. Then, from the definition of $T$ and $\tilde T$ [cf.\ Eq.~(\ref{eq-Bellman})] it follows that for a sequential SSP game satisfying Assumption~\ref{assum-ssp}, both players have stationary deterministic equilibriums policies. It also follows that for sequential games, we may replace the sets $\Ra$ and $\Rb$ in Assumption~\ref{assum-ssp} and Proposition~\ref{prp-3.1} by the sets of stationary deterministic policies of the two players.
\end{remark}
\smallskip

From Theorem~\ref{thm-opt-vipi} we obtain convergence of value and policy iteration for the finite-space game under Assumption~\ref{assum-ssp}. 
We will not focus on these algorithms in the rest of this paper, however. Instead, we will focus on a model-free algorithm called Q-learning, for computing the value function of the game. 
The algorithm is useful when the transition probabilities and expected one-stage costs are unknown or when the model is too complicated to have these parameters written down explicitly, but random transitions and transition costs can be observed or generated by a simulator.
The Q-learning algorithm may be viewed as a stochastic value iteration algorithm. 
Standard value iteration, however, computes $T J$ for some vector $J$ at each iteration. It would be difficult to do so in the model-free context, with only a few observations of state transitions, as can be seen from the expression of $T J$ in Eq.~({\ref{eq-Bellman}). 
The Q-learning algorithm will work not with the cost vector $J$ but with the so-called Q-factors and an associated dynamic programming equation, which is equivalent to 
the dynamic equation $J = T J$ by a change of variable (from $J$ to Q-factors). 
To prepare for the study of the Q-learning algorithm, let us explain this equation now.

\subsubsection*{Q-Factors and the Associated Dynamic Programming Equation}

Let $\Stc = \big\{ (i, u, v) \mid i \in \St,   u \in U(i), v \in V(i) \big\}$ be the state-and-control space. In the dynamic programming equation~(\ref{eq-Bellman}), which we repeat here:
$$ J(i) = (T J)(i) = \inf_{\da \in \bar U(i)} \sup_{\db \in \bar V(i)}  \sum_{u \in U(i)} \sum_{v \in V(i)} \da(u) \db(v) \Big( g(i, u, v) +  \sum_{j \in S} p_{ij}(u, v) J(j) \Big), \quad \forall \, i \in \St,$$
let us make a change of variable from $J$ to $Q = \big\{ Q(i,u,v) \mid (i,u,v) \in \Stc \big\}$ by letting
$$ Q(i,u,v) =  g(i, u, v) +  \sum_{j \in S} p_{ij}(u, v) J(j), \qquad (i,u,v) \in \Stc.$$
This gives an equation in terms $Q$: for all $(i,u,v) \in \Stc$,
\begin{align}
   Q(i,u,v) & = g(i, u, v) +  \sum_{j \in S} p_{ij}(u, v) (TJ)(j) \notag \\
   & = g(i, u, v) +  \sum_{j \in S} p_{ij}(u, v)   \inf_{\da \in \bar U(j)} \sup_{\db \in \bar V(j)}  \sum_{u' \in U(j)} \sum_{v' \in V(j)} \da(u') \db(v') Q(j, u', v'). \label{eq-Bellman-F0}
\end{align}
To simplify notation, we define the shorthand notation
\begin{equation} \label{eq-aveQ}
 \ave Q(i, \da, \db) = \sum_{u \in U(i)} \sum_{v \in V(i)} \da(u) \db(v)  Q(i, u, v)
\end{equation} 
for a given vector $Q$ and randomized decision rules $\da  \in \bar U(i)$, $\db \in \bar V(i)$ for a state $i$.
Then Eq.~(\ref{eq-Bellman-F0}) can be expressed concisely as
\begin{equation} \label{eq-Bellman-F}
 Q = F Q \quad \text{or} \quad Q(i, u, v) = (FQ)(i,u,v),    \quad \forall \,  (i,u,v) \in \Stc, 
\end{equation}
where the operator $F: \Re^{|\Stc|} \to \Re^{|\Stc|}$ is given by
\begin{equation} \label{eq-def-F}
(FQ)(i,u,v) := g(i,u,v) + \sum_{j \in S} p_{ij}(u, v)  \inf_{\da \in \bar U(j)} \sup_{\db \in \bar V(j)}  \ave Q(j, \da, \db), \qquad  (i,u,v) \in \Stc.
\end{equation}

We refer to the components of $Q$ as Q-factors. Equation $Q = F Q$ given by (\ref{eq-Bellman-F}) is the dynamic programming equation for Q-factors.
Since it is obtained from $J = T J$ by a change of a variable, any solution of $J = T J$ gives us a solution of $Q = F Q$.
Conversely, if in the equation $Q = F Q$ we change the variable $Q$ to $J$ by letting
$$ J(i)  =   \inf_{\da \in \bar U(i)} \sup_{\db \in \bar V(i)}  \ave Q(i, \da, \db), \qquad   \forall \, i \in \St,$$
then by a direct calculation, we get back the equation $J = T J$. Hence any solution of $Q = F Q$ gives us a solution of $J = T J$.
Furthermore, it can be verified using the definition of $F$ and $T$ that there is a one-to-one correspondence between the solutions of these two dynamic programming equations.
Using these facts, some optimality properties given in Prop.~\ref{prp-3.1} can be stated in terms of Q-factors as follows: 

\smallskip
\begin{cor}[Optimality Properties of Finite-Space SSP Games in terms of Q-factors] \label{cor-3.1}
For a finite-space SSP game satisfying Assumption~\ref{assum-ssp}, the dynamic programming equation~(\ref{eq-Bellman-F}) has a unique solution $Q^*$, which relates to the value function $J^*$ of the game by
\begin{align*}
   Q^*(i,u,v) & = g(i,u,v) +  \sum_{j \in S} p_{ij}(u, v) J^*(j), \qquad  \forall \, (i,u,v) \in \Stc, \\
   J^*(i)  & =   \inf_{\da \in \bar U(i)} \sup_{\db \in \bar V(i)}  \ave Q^*(i, \da, \db), \qquad \qquad \ \   \forall \, i \in \St.
\end{align*}   
Any stationary policies $\pla^* \in \Ra, \plb^* \in \Rb$ such that for every state $i$,
$$ \pla^*(\cdot \mid i) \in \argmin_{\da \in \bar U(i)} \sup_{\db \in \bar V(i)}  \ave Q^*(i, \da, \db), \qquad
\plb^*(\cdot \mid i ) \in \argmax_{\db \in \bar V(i)} \inf_{\da \in \bar U(i)}  \ave Q^*(i, \da, \db), $$
are optimal policies for the two players.
\end{cor}
\smallskip

As Cor.~\ref{cor-3.1} shows, for an SSP game satisfying Assumption~\ref{assum-ssp}, if we know $Q^*$, we can use it to compute the value function of the game and optimal policies of the two players, by solving for each state a matrix game defined by $Q^*$:  $\mathop{\text{minimax}}_{\da \in \bar U(i), \db \in \bar V(i)} \ave Q^*(i, \da, \db)$. These matrix game problems do not involve the parameters of the SSP game, which can be unknown in the learning context. 

Corollary~\ref{cor-3.1} also shows that under Assumption~\ref{assum-ssp}, $Q^*$ is the unique fixed point of the dynamic programming operator $F$, and therefore, since $F$ is also monotone and nonexpansive with respect to the sup-norm $\| \cdot \|_\infty$ by definition, the fixed point iteration 
$Q_{t+1} = F Q_t$ converges to $Q^*$ for any initial $Q_0$.  These properties are important for applying the Q-learning algorithm to compute $Q^*$.

\subsection{Q-Learning for SSP Games} \label{sec-alg-ql}

The Q-learning algorithm is an asynchronous stochastic iterative algorithm, and as mentioned earlier, it does not require the knowledge of the model parameters such as transition probabilities and expected one-stage costs. Instead, its computation is based on random state transitions and transition costs, which may be generated by a simulator or observed in a real learning environment.

We consider using Q-learning to compute the function $Q^*$ for a finite-space SSP game satisfying Assumption~\ref{assum-ssp}. 
Intuitively, one may view the algorithm as a stochastic version of damped fixed point iterations with the mapping $F$, i.e., iterations of the form $(1 - \g) Q + \g F Q$ for some stepsize parameter $\g$.
\footnote{The behavior of the Q-learning algorithm in practice is, however, much more complex than suggested by this simple view (in the context of MDP, its behavior can sometimes resemble policy iteration, for example). Such complexity can be attributed in part to various coordination schemes one can use with asynchronous and distributed computation.
This subject is beyond the scope of this paper, however.} 
The algorithm generates iteratively a sequence of Q-factor vectors, $\{Q_t\}$. Our main result is a proof that this sequence converges to $Q^*$ w.p.$1$ in a fairly general totally asynchronous computation setting.

To describe the algorithm, first recall that 
$$(FQ)(i,u,v) = g(i,u,v) + \sum_{j \in S} p_{ij}(u, v)  \inf_{\da \in \bar U(j)} \sup_{\db \in \bar V(j)}  \ave Q(j, \da, \db)$$
[cf.\ Eqs.~(\ref{eq-def-F}), (\ref{eq-aveQ})], so a damped fixed point iteration $Q_{t+1} = (1 - \g) Q_t + \g F Q_t$ will set the $(i,u,v)$-th component of $Q_{t+1}$ to be
$$ Q_{t+1}(i,u,v) = ( 1 - \g) Q_t(i,u,v) + \g \Big( g(i,u,v) + \sum_{j \in S} p_{ij}(u, v)  \inf_{\da \in \bar U(j)} \sup_{\db \in \bar V(j)}  {\ave Q}_t(j, \da, \db) \Big).$$
The Q-learning algorithm we describe next differs from the above iteration in several ways:\moveup\moveup
\begin{itemize}
\item[(i)] It is an asynchronous algorithm. At each iteration, it updates only a chosen subset of Q-factor components, keeping the rest unchanged.\moveup
\item[(ii)] Its computation can be distributed among multiple processors. Each Q-factor component can be updated by a separate processor, for example, and
communication delays are taken into account by allowing a processor to use outdated information in computation. 
In particular, for updating the $(i,u,v)$-th component at iteration $t$, the algorithm can use the Q-factor component $Q_{\tau}(j, \tilde u, \tilde v)$ computed at some iteration $\tau \leq t$, where $\tau$ can depend on both $(i,u,v)$ and $(j, \tilde u, \tilde v)$, reflecting the communication delay between the two associated processors.
In the algorithm, we will write these $\tau$ variables as $\tau_{\ell \tilde \ell}(t)$, for every pair of state-control triplets $\ell=(i,u,v), \tilde \ell=(j, \tilde u, \tilde v) \in \Stc$.
For each $\ell=(i,u,v) \in \Stc$, we will use the shorthand notation $Q_t^{(\ell)}$ to denote the Q-factor vector whose $(j, \tilde u, \tilde v)$-th component is given by:
\begin{equation} \label{eq-ql-notation1}
  Q_t^{(\ell)}(j, \tilde u, \tilde v) =  Q_{\tau_{\ell \tilde \ell}(t)}(j, \tilde u, \tilde v) \qquad \text{with} \ \ \tilde \ell = (j, \tilde u, \tilde v) \in \Stc. 
\end{equation}  
For $\ell = (i,u,v)$, we can view $Q_t^{(\ell)}$ as the ``local information'' that the $\ell$th processor uses for updating $Q_{t+1}(i,u,v)$.\moveup 
\item[(iii)] It is a model-free, stochastic approximation-based algorithm. Compared with the damped fixed point iteration $Q_{t+1}(i,u,v) = ( 1 - \g) Q_t(i,u,v) + \g (FQ_t^{(\ell)})(i,u,v)$ using possibly ``outdated'' information as just discussed,
the Q-learning iterate for $Q_{t+1}(i,u,v)$ uses, in place of $(FQ_t^{(\ell)})(i,u,v)$, an unbiased estimate of $(FQ_t^{(\ell)})(i,u,v)$ obtained through sampling state transitions randomly.
\end{itemize}

Let us describe now the Q-learning algorithm. The algorithm generates recursively a sequence $\{Q_t\}$ of Q-factor vectors.
At each iteration, it generates random state transitions, and the termination state $0$ and the zero total cost at that state appear explicitly in the calculation. 
For notational convenience, let us define for state $0$, the dummy control sets $U(0) = V(0) = \{0\}$ with $\bar U(0) = \bar V(0) = \P(\{0\})$,
and treat Q-factors as $(|\Stc|+1)$-dimensional vectors with $Q(0,0,0) = 0$. It will be taken for granted that $Q_t(0,0,0)=0$ for all $t$ and the variables $\tau_{\ell \tilde \ell}(\tilde \ell)$ for communications delays between $\ell \in \Stc$ and $\tilde \ell = (0,0,0)$ are (arbitrarily) defined. 
Given $\{Q_\tau, \tau \leq t\}$, the $t$th iteration of the algorithm computes $Q_{t+1}$ as follows. 

\medskip
\noindent {\bf Q-Learning Algorithm ($t$th iteration)} \hfill \vspace*{0.1cm}\\
\noindent For each state-control triplet $\ell = (i,u,v) \in \Stc$:\moveup\moveup
\begin{itemize}
\item[(a)] Let $\g_{t, \ell} \in [0,1]$ be a stepsize parameter. For each $\tilde \ell \in \Stc$, let $\tau_{\ell \tilde \ell}(t) \leq t$ be a nonnegative integer.\moveup
  \item[(b)] Generate a random transition from state $i$ with control $(u,v)$, and denote the successor state by $j_t^{\ell}$ (here $j_t^{\ell} \in \Ste$).
With $s$ being a shorthand for the state $j_t^{\ell}$, let
\begin{equation} \label{eq-Q}
Q_{t+1}(i,u,v) = ( 1 - \g_{t, \ell} ) Q_t(i,u,v) + \g_{t, \ell} \Big( \hat g(i, u, v, s) + \inf_{\da \in \bar U(s)} \sup_{\db \in \bar V(s)}  {\ave Q}^{(\ell)}_t(s, \da, \db) \Big).
\end{equation}
Here for $s \not = 0$ (i.e., $s$ is not the termination state), ${\ave Q}^{(\ell)}_t(s, \da, \db)$ is a shorthand notation for the weighted average of Q-factors,
$$ \sum_{\tilde u \in U(s)} \sum_{\tilde v \in V(s)} \da(\tilde u) \db(\tilde v) Q^{(\ell)}_t(s, \tilde u, \tilde v),$$
with $Q^{(\ell)}_t$ being the Q-factor vector given by Eq.~(\ref{eq-ql-notation1}). 
For $s = 0$, ${\ave Q}^{(\ell)}_t(s, \da, \db) = 0$ [which is also consistent with the preceding expression when we extend the definition in Eq.~(\ref{eq-ql-notation1}) to include $\tilde \ell = (0,0,0)$]. 
\end{itemize}
We note that the stepsize variables specify implicitly the subset of Q-factor components to be updated at iteration $t$. If $\g_{t,\ell}=0$, then $Q_{t+1}(\ell) = Q_t(\ell)$ and no computation is actually needed to carry out step (b).  The components with positive stepsizes, $\{\ell \in \Stc \mid \g_{t,\ell} > 0 \}$, are those for which the corresponding Q-factors are selected for an update.

The variables appearing in the Q-learning algorithm will be regarded as random variables on a common probability space $(\Omega, \F, \mbf{P})$.  
We require them to satisfy the following standard conditions for asynchronous Q-learning (cf.\ \cite{tsi94}). 
(In fact, without these conditions, the algorithm as just described is imprecise.)
Let $\{\mathcal{F}_t\}$ be an increasing sequence of sub-$\sigma$-fields of $\mathcal{F}$. (They represent the histories of the algorithm up to certain times.)

\begin{assumption}[Algorithmic Conditions] \label{assump-qlalg} \hfill
\moveup\moveup
\begin{itemize}
\item[(i)] $Q_0$ is $\mathcal{F}_0$-measurable.\moveup 
\item[(ii)] For every $\ell, \tilde \ell \in \Stc$ and $t \geq  0$, $\gamma_{t, \ell}$ and $\tau_{\ell \tilde \ell}(t)$ are $\mathcal{F}_t$-measurable. \moveup
\item[(iii)] For every $\ell = (i,u,v)  \in \Stc$ and $t \geq  0$, $j^{\ell}_t$ is $\mathcal{F}_{t+1}$-measurable and 
\begin{equation} \label{eq-cond-j}
 \mbf{P}(j^{\ell}_t = j \mid \F_t ) = p_{ij}(u,v), \qquad j \in \Ste. 
\end{equation} 
\item[(iv)] With probability $1$, 
\begin{equation} \label{eq-cond-delay}
   \lim_{t \to \infty} \tau_{\ell \tilde \ell}(t) = \infty,\qquad \forall \, \ell, \tilde \ell \in \Stc.
\end{equation}\moveup
\item[(v)] With probability $1$,
\begin{equation} \label{eq-cond-stepsize}
   \sum_{t \geq 0} \g_{t, \ell} = \infty, \qquad \sum_{t \geq 0} \g_{t, \ell}^2 < \infty, \qquad \forall \, \ell \in \Stc.
\end{equation}
\end{itemize}
\end{assumption}

Conditions (i)-(iii) are on the probabilistic dependence relations between the variables. They are naturally satisfied by the Q-learning algorithm in practice, when at each iteration, the values of stepsizes and communication delays are chosen before the random successor states are generated.
Condition (iv) is on the variables related to communication delays: it ensures that outdated information will eventually be purged by the algorithm, so it is a minimal requirement for totally asynchronous computation.
Condition (v) is a standard stepsize condition. 
It implies that every Q-factor component is updated infinitely often, which is certainly indispensable for the Q-learning algorithm to find $Q^*$ in the limit.

We have the following results regarding the convergence of the Q-learning algorithm given above.

\smallskip
\begin{theorem}[Boundedness of Q-Learning Iterates] \label{thm-boundql}
Consider a finite-space SSP game satisfying Assumption~\ref{assum-ssp}. Then under Assumption~\ref{assump-qlalg}(i)-(iii) and (v), for any given initial $Q_0$, the sequence $\{Q_t\}$ generated by the Q-learning algorithm (\ref{eq-Q}) is bounded w.p.$1$.
\end{theorem}

\smallskip
\begin{theorem}[Convergence of Q-Learning] \label{thm-convql}
Consider a finite-space SSP game satisfying Assumption~\ref{assum-ssp}. Then under Assumption~\ref{assump-qlalg}, for any given initial $Q_0$, the sequence $\{Q_t\}$ generated by the Q-learning algorithm (\ref{eq-Q}) converges w.p.$1$ to the unique solution $Q^*$ of the equation $Q = F Q$.
\end{theorem}
\smallskip

Theorem~\ref{thm-boundql} on the boundedness of $\{Q_t\}$ is our main result. Its proof will be the subject of the next section. 
Assuming it has been proved, the convergence of Q-learning stated in Theorem~\ref{thm-convql} follows by combining the boundedness result with a convergence theorem of Tsitsiklis \cite{tsi94}. We give this proof below.

\begin{proof}[Proof of Theorem~\ref{thm-convql}]
To analyze the convergence of the Q-learning iterates $Q_t$, we write them in a form that is standard for stochastic approximation-based analysis. 
For every $\ell = (i,u,v) \in \Stc$ and every $t \geq 0$, we express the iteration (\ref{eq-Q}) equivalently as
\begin{equation} \label{eq-equivQ}
   Q_{t+1}(i,u,v) = ( 1 - \g_{t, \ell} ) Q_t(i,u,v) + \g_{t, \ell} \big( F Q^{(\ell)}_t \big)(i,u,v) + \g_{t, \ell} \, w_{t, \ell}, 
\end{equation}
where $w_{t,\ell}$ is a noise term given by
$$ w_{t, \ell} = \hat g(i, u, v, s) + \inf_{\da \in \bar U(s)} \sup_{\db \in \bar V(s)}  {\ave Q}^{(\ell)}_t(s, \da, \db) - \big( F Q^{(\ell)}_t \big)(i,u,v), $$
and $s$ is a shorthand notation for the random successor state $j^{\ell}_t$.
Using Eq.~(\ref{eq-cond-j}) and the definition of the mapping $F$ [cf.\ Eq.~(\ref{eq-def-F})], direct calculation shows that the noise terms in the iteration (\ref{eq-equivQ}) satisfy that for every $\ell \in \Stc$ and $t \geq  0$,
$$   \E \big[ w_{t,\ell} \mid \mathcal{F}_t \big]  = 0, \qquad \text{w.p.$1$}, $$
and there exist deterministic constants $A$ and $B$, independent of $\ell$ and $t$, such that 
$$  \E \big[ w_{t,\ell}^2 \mid \mathcal{F}_t \big]  \leq A + B \max_{\ell' \in \Stc} \max_{\tau \leq t}  | Q_{\tau}(\ell') |^2, \qquad \text{w.p.$1$}. $$
Then, since under Assumption~\ref{assum-ssp}, $F$ has a unique fixed point $Q^*$ and is monotone and nonexpansive with respect to $\| \cdot \|_\infty$, a convergence theorem of Tsitsiklis \cite[Theorem 2]{tsi94} applies
and shows that $\{Q_t\}$ converges to $Q^*$ w.p.$1$, provided that $\{Q_t\}$ is bounded w.p.$1$. 
The desired convergence result then follows from Theorem~\ref{thm-boundql}.
\end{proof}

\begin{remark} \rm
We have set the stepsizes $\g_{t,\ell} \leq 1$ in this paper. Theorems~\ref{thm-boundql}, \ref{thm-convql} actually hold without this restriction, but in order to handle the general case of positive, possibly unbounded stepsizes, additional technical arguments are needed in the proofs, and such arguments can be found in the papers~\cite{yub_bd11,yb_ssp11}. To avoid the technical complication and repetition, in this paper we choose not to focus on general stepsizes.
\end{remark}

\begin{remark} \rm
As mentioned in the introduction section, using the O.D.E.-based analysis, Abounadi, Bertsekas and Borkar~\cite{abb02} established convergence for a class of asynchronous stochastic approximation algorithms involving nonexpansive mappings, and their results can be applied to the Q-learning algorithm for SSP games we consider. However, their asynchronous computation framework differs from the totally asynchronous computation framework we consider here.
A chief assumption in their framework is that all the components are updated comparatively often in the sense that $\liminf_{t \to \infty} m(t, \ell)/t > 0$ for all components $\ell$, where $m(t,\ell)$ is the number of times the $\ell$-th component has been updated up to time $t$.  (See also the related asynchronous schemes and their analyses in \cite{borkar-asyn,Borkar-Meyn-asyn}, \cite[Chap.\ 7]{borkar-sabook}.) If this and some other conditions on the stepsizes and communication delays are assumed to hold, the convergence result of~\cite{abb02} when applied in our context would lead to the conclusion that $Q_t$ tracks the scaled O.D.E.\ $\dot{Q} = \tfrac{1}{|\Stc|} (FQ - Q)$. By comparison, the totally asynchronous Q-learning algorithm considered here is generally not to be expected to have such kind of behavior, since it does not restrict how often a component should be selected for update. Correspondingly, the boundedness and convergence analyses of the algorithm for the totally asynchronous case also differ significantly from the O.D.E.-based analyses in the aforementioned works. 
\end{remark}

\section{Boundedness of Q-Learning Iterates} \label{sec-bd}

In this section we prove Theorem~\ref{thm-boundql} on the boundedness of Q-learning iterates for a finite-space SSP game satisfying Assumption~\ref{assum-ssp}.
The proof is long and uses a line of analysis devised earlier for bounding Q-learning iterates in single-player SSP problems (Yu and Bertsekas~\cite{yub_bd11}). 
After the proof of Theorem~\ref{thm-boundql}, which takes up Section~\ref{sec-bd-gen}, we include in Section~\ref{sec-bd-special} a short boundedness proof for a special case where the assumption on the game model is more restrictive than Assumption~\ref{assum-ssp} and the boundedness analysis is based on a contraction argument. 

\subsection{Boundedness Analysis for the General Case} \label{sec-bd-gen}

In this subsection, we prove the boundedness of Q-learning iterates stated in Theorem~\ref{thm-boundql}.
Assumption~\ref{assum-ssp} implies that there exist a policy $\bpla \in \Ra$ of player I and a policy $\bplb \in \Rb$ of player II that are essentially proper (Lemma~\ref{lma-app-properext}).
We will prove the lower boundedness of $\{Q_t\}$ by using the essential properness property of the policy $\bplb$ and by using the implications of this property on the single-player SSP problem for player I when player II plays the policy $\bplb$.
Due to symmetry, the same proof will also establish that $\{Q_t\}$ is bounded above w.p.$1$, by applying an identical argument to $\{ - Q_t\}$ and using the essential properness property of the policy $\bpla$.   

The proof consists of several steps, given in separate subsections. 
The main idea of the proof, reflected in the titles of these subsections, can be outlined as follows:\moveup
\begin{enumerate}
\item We relate $\{Q_t\}$ to a sequence $\{\hat Q_t\}$ of iterates that resembles Q-learning in the single-player SSP problem associated with the policy $\bplb$. We show that lower boundedness of $\{\hat Q_t\}$ implies lower boundedness of $\{Q_t\}$. (See Section \ref{sec-4.2.1}.)\moveup
\item  For any given positive scalar $\delta$, we construct an auxiliary sequence $\{\tilde Q_t\}$ such that (i) it is lower bounded w.p.$1$ if and only if $\{\hat Q_t\}$ is lower bounded w.p.$1$, and (ii) each component of $\tilde Q_t$ can be interpreted as the total cost of some policy in a time-inhomogeneous SSP problem in the ``$\delta$-neighborhood'' of the single-player SSP problem associated with the policy $\bplb$. (See Sections \ref{sec-4.2.2}-\ref{sec-4.2.5}.) These are the key steps of our proof.\moveup
\item We show that when $\delta$ is sufficiently small, the optimal total costs of all the single-player SSP problems in the aforementioned ``$\delta$-neighborhood'' can be bounded uniformly from below, and hence the auxiliary sequence $\{\tilde Q_t\}$ is bounded below w.p.$1$. (See Section~\ref{sec-4.2.6}.) This leads to the desired conclusion that $\{\hat Q_t\}$ and hence $\{Q_t\}$ are bounded below w.p.1, completing the proof.\moveup
\end{enumerate}
The auxiliary sequence-based arguments we use in this proof are first used in the boundedness analysis of Q-learning for single-player SSP problems~\cite{yub_bd11}. 

\subsubsection{Relate $\{Q_t\}$ to Q-learning type iterations in a single-player SSP problem} \label{sec-4.2.1}

To facilitate the analysis, 
we first reduce the question of lower boundedness of $\{Q_t\}$ to the question of lower boundedness of another process $\{{\hat Q}_t\}$, 
 which is defined on the same probability space as $\{Q_t\}$. 
The advantage of working with $\{\hat Q_t\}$ is that we can relate it to Q-learning like iterations for a single-player SSP that satisfies the SSP Model Assumption.

Let $\bplb \in \Rb$ be an essentially proper policy of player II; the existence of such a policy is ensured by Lemma~\ref{lma-app-properext} under Assumption~\ref{assum-ssp}.
To simplify notation, denote $\bplb_i = \bplb(\cdot \mid i)$ for every $i \in \Ste$ [note $\bplb_i \in \P(V(i))$]. 
We define an iteration similar to the Q-learning iteration~(\ref{eq-Q}), using the \emph{same random variables} (i.e., $\g_{t,\ell}$, $j_t^{\ell}$ and $\tau_{\ell \tilde \ell}(t)$, $\ell, \tilde \ell \in \Stc$) that appear in the Q-learning iteration~(\ref{eq-Q}). 
In particular, let $\hat Q_0 = Q_0$ and for $t \geq 0$ and for every $\ell = (i,u,v) \in \Stc$, let
\begin{equation} \label{eq-bdbl-Q}
{\hat Q}_{t+1}(i,u,v) = ( 1 - \g_{t, \ell} ) {\hat Q}_t(i,u,v) + \g_{t, \ell} \Big( \hat g(i, u, v, s) +  \inf_{\da \in \bar U(s)}  {\ave {\hat Q}}^{(\ell)}_t (s, \da, \bplb_s ) \Big),
\end{equation}
where 
$s$ is a shorthand for the successor state $j^{\ell}_t$, and the expression ${\ave {\hat Q}}^{(\ell)}_t (s, \da, \bplb_s )$ 
denotes a weighted average of Q-factors given by
$$ {\ave {\hat Q}}^{(\ell)}_t (s, \da, \bplb_s ) = \sum_{\tilde u \in U(s)} \sum_{\tilde v \in V(s)} \da(\tilde u) \bplb_s(\tilde v) {{\hat Q}}^{(\ell)}_t(s, \tilde u, \tilde v)$$
with ${{\hat Q}}^{(\ell)}_t$ being the vector whose $\tilde \ell$th component for $\tilde \ell \in \Stc \cup \{(0,0,0)\}$ is given by ${{\hat Q}}^{(\ell)}_t(\tilde \ell) = {\hat Q}_{\tau_{\ell \tilde \ell}(t)}(\tilde \ell)$, similar to the definition of $Q^{(\ell)}_t$ given by Eq.~(\ref{eq-ql-notation1}). (By default ${\hat Q}_\tau(0,0,0) = 0$ for all $\tau$.)
The iteration (\ref{eq-bdbl-Q}) differs from the Q-learning iteration~(\ref{eq-Q}) in that instead of maximizing over $\db \in \bar V(s)$, we fix $\db$ at $\bplb_s$.

\smallskip
\begin{lemma} \label{lma-bdbl-Q}
If $\{{\hat Q}_t\}$ is bounded below w.p.$1$, so is $\{Q_t\}$.
\end{lemma}

\begin{proof}
We show by induction that $Q_{t} \geq {\hat Q}_{t}$ for all $t$. For $t = 0$, this holds since $\hat Q_0 = Q_0$ by definition. 
Suppose that for some $t \geq 0$, the desired relation holds for all $\tau \leq t$.
Then, for every $\ell=(i,u,v) \in \Stc$,  using Eqs.~(\ref{eq-Q}), (\ref{eq-bdbl-Q}), the induction hypothesis, and the fact that $\g_{t,\ell} \in [0,1]$, we have that
\begin{align*}
   Q_{t+1}(i,u,v) & \geq ( 1 - \g_{t, \ell} )  Q_t(i,u,v)  + \g_{t, \ell} \Big( \hat g(i, u, v, s) + \inf_{\da \in \bar U(s)}  {\ave Q}^{(\ell)}_t (s, \da, \bplb_s ) \Big) \\
      & \geq ( 1 - \g_{t, \ell} )  {\hat Q}_t(i,u,v)  + \g_{t, \ell} \Big( \hat g(i, u, v, s) +\inf_{\da \in \bar U(s)}  {\ave {\hat Q}}^{(\ell)}_t (s, \da, \bplb_s )   \Big) 
      = {\hat Q}_{t+1}(i,u,v),
 \end{align*}
where $s$ is a shorthand for the successor state $j^{\ell}_t$.
This completes the induction and establishes that $Q_{t} \geq {\hat Q}_{t}$ for all $t$. Hence $\{Q_t\}$ is bounded below w.p.$1$ if $\{{\hat Q}_t\}$ is so.
\end{proof}
\smallskip

By the preceding lemma, in order to establish the lower boundedness of the Q-learning iterates $\{Q_t\}$, it is sufficient to prove that the sequence $\{\hat Q_t\}$ defined above is bounded below w.p.$1$. 
The iterates $\{\hat Q_t\}$ are similar to Q-learning iterates in an MDP. Our goal now is to make this connection more precise so that we can apply the results or proof techniques developed for analyzing Q-learning in single-player problems to bound $\{\hat Q_t\}$ from below.
To this end, let us examine the single-player problem faced by player I when player II plays the essentially proper policy $\bplb$.
We will call this single-player SSP problem \sspa. For later use, we will augment its state space to include the set $R$ also. Here is the precise definition of \sspa.

\smallskip
\begin{definition} \label{def-sspa}
\sspa\ denotes the following single-player SSP problem: 
 \moveup\moveup
\begin{enumerate}
\item[(1)] The state space is $\Ste \cup \Stc$, with state $0$ being a cost-free termination state.\moveup\moveup
\item[(2)] From a state $\ell = (i,u,v) \in \Stc$, the system transitions to a state $j \in \Ste$. The transition is uncontrolled and occurs with probability $p_{ij}(u,v)$, and the expected one-stage cost is $g(i,u,v)$. \moveup\moveup
\item[(3)] For a state $i \in \St$, the control set is $U(i)$, and for each $u \in U(i)$, the system transitions to a state $j \in \Ste$ with probability
\begin{equation} \label{eq-sspa-transprob}
  p_{{\bplb},ij}(u) = \sum_{v \in V(i)} \bplb(v \mid i) \, p_{ij}(u,v),
\end{equation}  
and the expected one-stage cost is
\begin{equation} \label{eq-sspa-transcost} 
 g_{\bplb}(i, u) =  \sum_{v \in V(i)} \bplb(v \mid i) \, g(i,u,v).
\end{equation} 
\end{enumerate}
\end{definition}
\smallskip

Because $\bplb$ is an essentially proper policy of the SSP game (cf.\ Definition~\ref{def-essproper-sspgame}), we can show that the single-player problem \sspa\ satisfies the single-player SSP Model Assumption (cf.\ the discussion preceding Definition~\ref{def-essproper-sspgame}). 
Let $\SD$ ($\SR$) denote the set of stationary deterministic (randomized) policies in \sspa.

\smallskip
\begin{lemma} \label{lma-sspa}
\sspa\ satisfies the SSP Model Assumption; that is,  there exists a proper policy in $\SD$, and
every improper policy in $\SD$ incurs infinite cost for at least one initial state.
(Here proper and improper policies are as defined in Definition~\ref{def-proper-ssp} for a single-player SSP problem.)
\end{lemma}

\begin{proof}
The system dynamics of \sspa\ described in Definition~\ref{def-sspa}(2)-(3) shows that to prove the lemma, it suffices to consider only those states in $\Ste$ and prove that \sspa\ restricted to $\Ste$ satisfies the SSP model Assumption. 
Let us simply call this restricted problem \sspa\ in the proof below. 
Since $\bplb$ is an essentially proper policy of a finite-space SSP game that satisfies Assumption~\ref{assum-ssp}, by Definition~\ref{def-essproper-sspgame}(b), \sspa\ has the following properties: there exists a proper policy in $\SR$, and every improper policy in $\SR$ incurs infinite cost for at least one initial state.
Hence, to prove the lemma, we need to show that \sspa\ has a proper policy in $\SD$. 

We claim that if there exists a proper policy in $\SR$, then there must exist a proper policy in $\SD$. This follows from the relation between the limiting average state-action frequency of a stationary randomized policy and the set of the limiting average state-action frequencies of all stationary deterministic policies, in a finite-space MDP. (For the definition of these limiting frequencies, see \cite[Section 8.9.1]{puterman94}.) In particular, consider any initial state distribution $\alpha$ such that $\alpha(i) > 0$ for all $i \in \St$. Let $\pla \in \SR$ be a proper policy. Let $y_{\pla,\alpha} = \{ y_{\pla,\alpha}(i,u) \mid i \in \Ste,  u \in U(i) \}$ denote the limiting average state-action frequency of $\pla$ for the initial state distribution $\alpha$ (here the control set for the termination state is set to be $U(0) = \{0\}$.) By \cite[Theorem 8.9.3, p.~400]{puterman94}, $y_{\pla,\alpha}$ lies in the convex hull of the limiting average state-action frequencies of stationary deterministic policies for the initial distribution $\alpha$. 
Since $\pla$ is proper, the termination state $0$ is reached w.p.$1$ for all initial states in $\St$, and consequently, $y_{\pla,\alpha}$ is the vector with $y_{\pla,\alpha}(0,0) = 1$ for the termination state $i=0$ and with $y_{\pla,\alpha}(i,u) = 0$ for $(i,u) \not= (0,0)$. This vector must be an extreme point of the convex hull just mentioned (which is a subset of probability distributions on $\{(i,u) \mid i \in \Ste, u \in U(i)\}$). Therefore, there exists some $\pla_{\text{det}} \in \SD$ whose limiting average state-action frequency for the initial distribution $\alpha$ equals $1$ at $(i,u) = (0,0)$. 
Since $\alpha(i) > 0$ for all $i \in \St$, this implies that the termination state $0$ is reached w.p.$1$ for all initial states in $\St$ under the deterministic policy $\pla_{\text{det}}$. Hence $\pla_{\text{det}}$ is a proper policy in $\SD$. This proves our claim. 
\end{proof}

For an SSP satisfying the SSP Model Assumption, the classical Q-learning algorithm generates a sequence of iterates that is bounded w.p.$1$, as proved by Yu and Bertsekas~\cite{yub_bd11}. 
The iterates $\{\hat Q_t\}$ defined by Eq.~(\ref{eq-bdbl-Q}) are similar to the classical Q-learning iterates, except for a small difference: in iteration~(\ref{eq-bdbl-Q}), the minimization over the controls at the successor state is done after taking weighted averages of Q-factors (weighted according to $\bplb$), whereas there is no such averaging in classical Q-learning.
This difference is mostly algebraic, however. 
Our subsequent proof of the lower boundedness of $\{\hat Q_t\}$ follows essentially the lower boundedness proof  given in \cite[Section 3.3]{yub_bd11} for classical Q-learning.

\subsubsection{Auxiliary sequence $\{\tilde Q_t\}$} \label{sec-4.2.2}

We proceed to prove that $\{\hat Q_t\}$ given by iteration~(\ref{eq-bdbl-Q}) is bounded below w.p.$1$ for any given initial $\hat Q_0$.
We will do so by introducing yet another process $\{\tilde Q_t \}$ on the same probability space. The construction of this new process will be the key to our proof.

To this end, let us replace the inf operation in iteration~(\ref{eq-bdbl-Q}) and
write iteration~(\ref{eq-bdbl-Q}) equivalently as follows. For every $\ell=(i,u,v) \in \Stc$ and $t \geq 0$,
\begin{equation} \label{eq-bdbl-Q1}
{\hat Q}_{t+1}(i,u,v) = ( 1 - \g_{t, \ell} ) {\hat Q}_t(i,u,v) + \g_{t, \ell} \Big( \hat g(i, u, v, s) +   {\ave {\hat Q}}^{(\ell)}_t(s, u^{\ell}_t, \bplb_s) \Big),
\end{equation}
where $s$ is a shorthand for the successor state $j^{\ell}_t$, and $u^\ell_t$ is a control such that
$$ u^{\ell}_t  \in \argmin_{ \tilde u \in U(s)}  {\ave {\hat Q}}^{(\ell)}_t(s, \tilde u, \bplb_s),$$
where the expression ${\ave {\hat Q}}^{(\ell)}_t(s, \tilde u, \bplb_s)$ for $\tilde u \in U(s)$ denotes the weighted average of the Q-factors:
$$ {\ave {\hat Q}}^{(\ell)}_t(s, \tilde u, \bplb_s)  = \sum_{\tilde v \in V(s)} \bplb_s(\tilde v) \, {\hat Q}^{(\ell)}_t(s, \tilde u, \tilde v). $$

Now consider an auxiliary sequence $\{\tilde Q_t\}$ of the following form.  \emph{Given} some integer $t_0$ and Q-factor vector $\tilde Q_{t_0}$, let
\begin{equation} \label{eq-tQ-0}
  \tilde Q_t = \tilde Q_{t_0}, \qquad t \leq t_0,
\end{equation}  
and let $\tilde Q_{t+1}$, $t \geq t_0$, be defined by the recursion:  for every $(i,u,v) \in \Stc$, 
\begin{equation} \label{eq-tQ}
 \tilde Q_{t+1}(i,u,v) = ( 1 - \g_{t, \ell} ) {\tilde Q}_t(i,u,v) + \g_{t, \ell} \Big( \hat g(i, u, v, s) +   {\ave {\tilde Q}}^{(\ell)}_t(s, u_t^{\ell}, \bplb_s) \Big),
\end{equation} 
where $s$ is a shorthand for the successor state $j^{\ell}_t$, and the expression
${\ave {\tilde Q}}^{(\ell)}_t(s, u_t^\ell , \bplb_s)$ is a shorthand for a weighted average of Q-factors, defined similarly to the notation ${\ave {\hat Q}}^{(\ell)}_t(s, \tilde u, \bplb_s)$ above:
$$ {\ave {\tilde Q}}^{(\ell)}_t(s, u_t^{\ell}, \bplb_s)  = \sum_{\tilde v \in V(s)} \bplb_s(\tilde v) \, {\tilde Q}^{(\ell)}_t(s, u_t^{\ell}, \tilde v),  $$
with ${{\tilde Q}}^{(\ell)}_t$ representing the vector of Q-factors whose components are given by 
${{\tilde Q}}^{(\ell)}_t(\tilde \ell) = {\tilde Q}_{\tau_{\ell \tilde \ell}(t)}(\tilde \ell), \tilde \ell \in \Stc \cup \{(0,0,0)\}$ 
[by default $\tilde Q_t(0,0,0)= {{\tilde Q}}^{(\ell)}_t(0,0,0) = 0$ for all $t$]. 
Most importantly, the variables $\g_{t,\ell}$, $j_t^{\ell}, u_t^\ell$, and $\tau_{\ell \tilde \ell}(t)$ where $\ell, \tilde \ell \in \Stc, t \geq 0$, in the definition (\ref{eq-tQ}) for $\{\tilde Q_{t}\}$ are the same random variables that appear in the iteration~(\ref{eq-bdbl-Q1}) that defines $\{\hat Q_t\}$. 

\smallskip
\begin{lemma} \label{lma-bdbl-hQ}
Consider any sample path. 
Then for any values of $t_0$ and $\tilde Q_{t_0}$, $\{\hat Q_t\}$ is bounded below if and only if $\{\tilde Q_t\}$ given by Eqs.~(\ref{eq-tQ-0})-(\ref{eq-tQ}) is bounded below. 
\end{lemma}

\begin{proof}
For every $\ell=(i,u,v) \in \Stc$ and $t \geq t_0$, using Eqs.~(\ref{eq-bdbl-Q1}), (\ref{eq-tQ}), and the fact that $\g_{t, \ell} \in [0,1]$, we have that 
\begin{align*}
 \big| {\hat Q}_{t+1}(i,u,v) - \tilde Q_{t+1}(i,u,v) \big|  & \leq ( 1 - \g_{t, \ell} ) \big| {\hat Q}_{t}(i,u,v) - \tilde Q_{t}(i,u,v) \big|  \\
 & \quad \ + 
   \g_{t, \ell} \sum_{\tilde v \in V(s)} \bplb_s(\tilde v)  \big|  {\hat Q}^{(\ell)}_t(s, u_t^{\ell}, \tilde v) - {\tilde Q}^{(\ell)}_t (s, u_t^{\ell}, \tilde v) \big| \\   
& \leq \max_{\tau \leq t} \| \hat Q_\tau - \tilde Q_\tau \|_\infty,
\end{align*}
where $s = j^\ell_t$.
This implies that for all $t \geq t_0$,
$$ \max_{\tau \leq t+1} \| \hat Q_\tau - \tilde Q_\tau \|_\infty \leq \max_{\tau \leq t} \| \hat Q_\tau - \tilde Q_\tau \|_\infty.$$
Hence, on a sample path, $\{\hat Q_t\}$ is bounded below if and only if $\{\tilde Q_t\}$ is bounded below.
\end{proof}
\smallskip

The sequence $\{\tilde Q_t\}$ is more convenient  to work with than $\{\hat Q_t\}$, because by Lemma~\ref{lma-bdbl-hQ} we have the freedom to choose for each sample path the initial time $t_0$ and initial value $\tilde Q_{t_0}$ so that the resulting sequence $\{\tilde Q_t\}$ has a certain desirable structure.
In the next step of the proof, we will make such a choice that will equate $\{\tilde Q_t\}$ to the costs in certain single-player SSP problems ``neighboring'' \sspa, in the sense that the parameters of these SSP problems lie close to those of \sspa.

Before we proceed, we need some notation and definitions for various neighborhoods of the model parameters, which we will use throughout the rest of the proof. 

\subsubsection{Some notation and definitions} \label{sec-4.2.3}

As before, for a finite set $A$, we denote by $\mathcal{P}(A)$ the set of probability distributions on $A$,
and for $a \in A$ and $p \in \mathcal{P}(A)$, we write $p(a)$ for the probability of $a$ under $p$.  
The support of $p$, denoted $\supp(p)$, is the set $\{ a \in A \mid p(a)  \not= 0 \}$.
For $p_1, p_2 \in \mathcal{P}(A)$, we write $p_1 \ll p_2 $ if $p_1$ is absolutely continuous with respect to $p_2$, that is, $\supp(p_1) \subset \supp(p_2)$. 

We use the following notation to represent the neighborhoods of the transition probability and one-stage cost parameters of \sspa\ within certain affine subspaces. (The parameters of \sspa\ are defined in Definition~\ref{def-sspa} and will be referred to below.)\moveup\moveup
\begin{itemize}
\item For each $\ell = (i,u, v) \in \Stc$, let $\mbf{p}^{\ell}_o \in \mathcal{P}(\Ste)$ denote the transition probability distribution at state $\ell$ in \sspa, that is,
$\mbf{p}^{\ell}_o(j) = p_{ij}(u, v), \forall  j \in \Ste.$
For each $\delta > 0$, define $\A_{\delta}(\mbf{p}_o^{\ell})$ to be the set of probability distributions that are not only in the $\delta$-neighborhood of $\mbf{p}_o^{\ell}$ but also absolutely continuous with respect to $\mbf{p}^{\ell}_o$, i.e., 
$$\A_{\delta}(\mbf{p}_o^{\ell})  = \big\{ \, \mbf{d} \in \mathcal{P}(\Ste) \, \big| \,  | \mbf{d}(j) - \mbf{p}^{\ell}_o(j) | \leq \delta, \  \forall \, j \in \Ste, \ \text{and} \ \mbf{d} \ll \mbf{p}^{\ell}_o  \big\}. \moveup$$
\item Denote $\Stce^1 = \{(i,u) \mid i \in \Ste, u \in U(i) \}$ where $U(0) = \{0\}$ denotes the (dummy) control set for the termination state $0$. (This is the set of state-control pairs for player I.)
For each $(i,u) \in \Stce^1$, let $\mbf{p}^{iu}_{\bplb} \in \mathcal{P}(\Ste)$ denote the transition probability distribution at state $i$ with control $u$ in \sspa:
$\mbf{p}^{iu}_{\bplb}(j) = p_{{\bplb},ij}(u), \forall   j \in \Ste.$
For each $\delta > 0$, define $\A_{\delta}(\mbf{p}_{\bplb}^{iu})$ to be the subset of distributions in the $\delta$-neighborhood of $\mbf{p}^{iu}_{\bplb}$ that are absolutely continuous with respect to $\mbf{p}^{iu}_{\bplb}$:
$$\A_{\delta}(\mbf{p}_{\bplb}^{iu})  = \big\{ \, \mbf{d} \in \mathcal{P}(\Ste) \, \big| \,  | \mbf{d}(j) - \mbf{p}^{iu}_{\bplb}(j) | \leq \delta, \  \forall \, j \in \Ste, \ \text{and} \ \mbf{d} \ll \mbf{p}^{iu}_{\bplb}  \big\}.$$
In particular, for $(i,u) = (0,0)$, $\mbf{p}^{iu}_{\bplb}(0) = 1$ and  $\A_{\delta}(\mbf{p}^{iu}_{\bplb}) = \big\{ \mbf{p}^{iu}_{\bplb} \big\}$.\moveup
\item Let $g = \{g(i,u,v) \mid (i,u,v) \in \Stc \}$ be the vector of expected one-stage costs for the states in $\Stc$ in \sspa.
Define $\B_{\delta}(g)$ to be the $\delta$-neighborhood of $g$:
with $\mbf{c} = \{ c(i,u,v) \mid  (i,u,v) \in \Stc \}$,
$$\B_{\delta}(g)  = \big\{ \, \mbf{c} \, \big| \,    | c(i,u,v) - g(i,u,v) | \leq \delta, \  \forall \, (i,u,v) \in \Stc \big\}.$$
\item Similarly,  let $g_{\bplb} = \{ g_{\bplb}(i,u) \mid  (i, u) \in \Stce^1 \}$ be the vector of expected one-stage costs for the state-control pairs in $\Stce^1$ in \sspa.
Define $\B_{\delta}(g_{\bplb})$ to be the intersection of the $\delta$-neighborhood of $g_{\bplb}$ with a subspace:
with $\mbf{c} = \{ c(i,u) \mid  (i, u) \in \Stce^1 \}$, 
$$\B_{\delta}(g_{\bplb})  = \big\{ \, \mbf{c} \, \big| \,  | c(i,u) - g_{\bplb}(i,u) | \leq \delta, \  \forall \, (i, u) \in \Stce^1 , \ \text{and}  \ c(0,0) = 0  \, \big\}. \moveup$$
\end{itemize}
For brevity, we will simply call the above sets $\A_{\delta}(\mbf{p}_o^{\ell})$, $\A_{\delta}(\mbf{p}_{\bplb}^{iu})$, $\B_{\delta}(g)$ and $\B_{\delta}(g_{\bplb})$ the $\delta$-neighborhoods of the respective parameters of \sspa.

\subsubsection{Choose $t_0$ and initial $\tilde Q_{t_0}$ for a sample path and $\delta > 0$} \label{sec-bl-init}

To initialize the auxiliary sequence $\{\tilde Q_t\}$ defined by Eqs.~(\ref{eq-tQ-0})-(\ref{eq-tQ}), we will choose time $t_0$ and vector $\tilde Q_{t_0}$ based on the information of an entire sample path. 
First, we define two random sequences on the same probability space as the process $\{\hat Q_t\}$: a sequence $\{\G_t\}$ of one-stage cost vectors,
and a sequence $\{\mbf{q}_t\}$ of collections of probability distributions on $\Ste$. 
They will be used to determine the values of $t_0$ and $\tilde Q_{t_0}$  on a sample path, for any chosen $\delta > 0$.

The sequence $\{\G_t\}$ can be related to the empirical one-stage costs and is defined recursively as follows. For $t \geq 0$,
\begin{equation} \label{eq-definitg}
\G_{t+1}(i,u, v) = \big( 1 - \gamma_{t,\ell} \big) \, \G_{t}(i,u,v) + \gamma_{t, \ell}  \, \hat g \big(i,u,v, j^{\ell}_t \big),  \qquad \forall \, \ell = (i,u, v) \in \Stc,
\end{equation}
with $\G_0(\cdot) \equiv 0$ for $t=0$.
By the standard theory of stochastic approximation (see e.g., \cite[Prop.\ 4.1 and Example 4.3, p.\ 141-143]{BET} or \cite{ky03,borkar-sabook}), 
Eqs.~(\ref{eq-cond-j}) and (\ref{eq-cond-stepsize}) imply that
\begin{equation}  \label{eq-convinitg}
     \G_t(i,u,v) \, \asto \, g(i,u, v),  \qquad \forall \, (i,u,v) \in \Stc,
\end{equation}
where ``a.s.'' stands for ``almost surely,'' ``w.p.$1$.''

The sequence $\{ \mbf{q}_t \}$ can be related to empirical  frequencies of state transitions and is defined recursively as follows. 
It has $| \Stc|$ component sequences, $\{\mbf{q}^{\ell}_t\}$, $\ell \in \Stc$.
For each $\ell=(i,u,v) \in \Stc$, let 
$$ \mbf{q}^{\ell}_0 \in \mathcal{P}(\Ste), \quad \mbf{q}^{\ell}_0 \ll \mbf{p}^{\ell}_o,$$ 
and let
\begin{equation} \label{eq-definitp}
 \mbf{q}^{\ell}_{t+1} = \big( 1 - \gamma_{t,\ell} \big)\, \mbf{q}^{\ell}_t +  \gamma_{t,\ell} \, \mbf{e}_{j^{\ell}_t}, \qquad t \geq 0,
\end{equation}
where for $j \in \Ste$, $\mbf{e}_{j} \in \mathcal{P}(\Ste)$ denotes the distribution with $\mbf{e}_j(j) = 1$. Then for all $\ell = (i,u,v) \in \Stc$, we have
\begin{equation}   \label{eq-convinitp}
  \mbf{q}^{\ell}_t \, \asto \, \mbf{p}^{\ell}_o  \qquad \text{and} \qquad    \mbf{q}^{\ell}_t \ll \mbf{p}^{\ell}_o \ \ \ \text{w.p.$1$}, \ \ \ \forall \, t \geq 0, 
\end{equation}
where the first relation follows from Eqs.~(\ref{eq-cond-j}), (\ref{eq-cond-stepsize}) and the standard theory of stochastic approximation, and the second relation follows from the fact that $j^{\ell}_t$ is a random successor state of state $i$ with controls $(u,v)$.

Equations~(\ref{eq-convinitg}), (\ref{eq-convinitp}) indicate that the sequences $\{\G_t(i,u,v) \}$ and $\{\mbf{q}^{\ell}_t \},$ $\ell = (i,u,v) \in \Stc$, converge to the corresponding one-stage cost and transition probability parameters of \sspa. We then obtain the following lemma, with which we will choose the initial time $t_0$.

\smallskip
\begin{lemma} \label{lma-init-tQ}
There exists a set of probability one on which, given any $\delta > 0$, there is a path-dependent time $t_0$ such that
\begin{equation}
 \mbf{q}^{\ell}_t \in \A_\delta( \mbf{p}^{\ell}_o),   \qquad  \G_t \in \B_\delta(g),  \qquad \forall \, \ell \in \Stc,  \   t \geq t_0.
\end{equation} 
\end{lemma}
\smallskip

In the rest of the proof, we consider \emph{any sample path} from the set of probability one given in Lemma~\ref{lma-init-tQ}.
For any given $\delta > 0$, we choose $t_0$ given in Lemma~\ref{lma-init-tQ} to be the initial time of the auxiliary sequence $\{ \tilde Q_t \}$. (Note that $t_0$ depends on the entire path and hence so does $\tilde Q_t$ for all $t$.)

We now define the initial $\tilde Q_{t_0}$. Let us fix some policy $\mu \in \Ra$ that is proper for the single-player problem \sspa. 
(Such a policy exists because $\bplb$ is an essentially proper policy of the game and \sspa\ satisfies the SSP Model Assumption; cf.\ the proof of Lemma~\ref{lma-sspa}, Section~\ref{sec-4.2.1}.)
Associate with $t_0$ and each $\ell=(i,u,v) \in \Stc$ a time-inhomogeneous Markov chain $(i_0, u_0,v_0), (i_1, u_1), (i_2, u_2), \ldots$ with time-varying one-stage costs as follows:\moveup\moveup
\begin{itemize}
\item The initial state of this Markov chain is $(i_0, u_0, v_0) = (i,u,v)$. The state space for time $k=0$ is $\Stc$ and for time $k \geq 1$ is $S_0 \times \U$ where $\U = \cup_{i \in \Ste} U(i)$ (the control space of player I).\moveup
\item The probability distribution of this Markov chain, 
denoted $\mbf{P}^{\ell}_{t_0}$, is defined by time-varying transition probabilities: 
for all $(\bar i, \bar u), (\bar j, \bar w) \in \Stce^1$,
\begin{align}
 \mbf{P}^{\ell}_{t_0} \big( i_1 = \bar j, u_1 = \bar w \mid i_{0} = i, u_{0} = u, v_0=v \big) & = \mbf{q}^{\ell}_{t_0}(\bar j) \cdot \mu ( \bar w \mid \bar j),  \qquad \text{for} \ k = 1,  \label{eq-Pt0-a} \\
 \mbf{P}^{\ell}_{t_0} \big( i_k = \bar j, u_k = \bar w \mid i_{k-1} = \bar i, u_{k-1} = \bar u \big) & = \mbf{p}_{\bplb}^{\bar i \bar u}(\bar j) 
 \cdot \mu ( \bar w \mid \bar j),  \qquad \text{for} \ k \geq 2. \label{eq-Pt0-b}
\end{align}
The transition probabilities at those $(\bar i, \bar u) \not\in \Stce^1$ can be defined arbitrarily because the chain has zero probability to visit such state-control pairs at any time, in view of the fact that $\mu$ is a policy for \sspa.\moveup
\item Define time-varying one-stage cost functions $g_0^{\ell, t_0}: \Stc \mapsto \Re$ and $g_k^{\ell, t_0}: \Stce^1 \mapsto \Re$, $k \geq 1$, to be
\begin{equation}
  g^{\ell, t_0}_0 = \G_{t_0}, \qquad g^{\ell, t_0}_k = g_{\bplb}, \qquad  \ k \geq 1. \label{eq-gt0}
\end{equation}  
For $k \geq 1$, we extend $g^{\ell,t_0}_k$ to $\Ste \times \U$ by defining its values outside the domain $\Stce^1$ to be $+ \infty$, and we will treat $0 \cdot  \infty = 0$. This convention will be followed throughout. 
\end{itemize}

We now define for every $\ell = (i,u,v) \in \Stc$,
\begin{equation} \label{eq-tQ0}
 \tilde Q_{t_0} (i,u, v) = g^{\ell, t_0}_0(i, u, v) + \E^{ \mbf{P}^{\ell}_{t_0}} \Big[ \, \sum_{k=1}^\infty  g^{\ell, t_0}_k(i_k, u_k) \,  
 \Big], 
\end{equation}
where $\mbf{P}^{\ell}_{t_0}$ in the superscript indicates that the expectation is taken with respect to it.
The above expectation is well-defined and finite, and furthermore, the order of summation and expectation can be exchanged, i.e., 
$$\tilde Q_{t_0} (i,u,v) = g^{\ell, t_0}_0(i, u, v) + \sum_{k=1}^\infty  \E^{ \mbf{P}^{\ell}_{t_0}} \Big[ \,  g^{\ell, t_0}_k(i_k, u_k)  \, 
\Big].$$
This is because according to the preceding definition of the Markov chain associated with $t_0$ and $\ell=(i,u,v)$, under $\mbf{P}^{\ell}_{t_0}$,  from time $1$ onwards, the process $\{(i_k, u_k), k \geq 1\}$ evolves and incurs costs as in \sspa\  under the proper policy $\mu$ [cf.\ Eqs.~(\ref{eq-Pt0-a})-(\ref{eq-gt0})],  
and consequently, $\sum_{k=1}^\infty  | g^{\ell, t_0}_k(i_k, u_k)|$ is finite almost surely and its expectation is finite with respect to $\mbf{P}^{\ell}_{t_0}$. 

The definition of $\tilde Q_{t_0}$ above has two key properties:\moveup\moveup 
\begin{itemize}
\item[(i)] Each component $\tilde Q_{t_0}(i,u,v)$ equals the expected total cost of some randomized Markov policy(which is $\mu$ here) in a time-inhomogeneous (single-player) SSP problem.\moveup
\item[(ii)]  The parameters of that SSP problem, i.e., transition probabilities 
and one-stage costs, 
all lie in the $\delta$-neighborhoods $\A_{\delta}(\mbf{p}_o^{\ell})$, $\A_{\delta}(\mbf{p}_{\bplb}^{iu})$, $\B_{\delta}(g)$, $\B_{\delta}(g_{\bplb})$ of the corresponding parameters of \sspa.\moveup\moveup
\end{itemize}
We now show that these properties are preserved in $\tilde Q_t, t \geq t_0$ defined by iteration~(\ref{eq-tQ}). 

\subsubsection{Interpret $\{\tilde Q_t\}$ as total costs in certain SSP problems neighboring \sspa} \label{sec-4.2.5}

The next lemma states precisely the interpretation we need of the auxiliary sequence $\{\tilde Q_t\}$ resulting from the preceding choice of $t_0$ and $\tilde Q_{t_0}$. 
Briefly speaking, each component of $\tilde Q_t, t \geq t_0,$ is equal to the expected total cost of a randomized Markov policy (represented by $\{\mu^{\ell,t}_k, k \geq 1\}$ below) in a time-inhomogeneous SSP problem whose parameters (transition probabilities and one-stage costs, represented by $\{p^{\ell,t}_{k}, g^{\ell,t}_k, k \geq 0\}$ below) lie in the $\delta$-neighborhoods of the corresponding parameters of \sspa. 

\smallskip
\begin{lemma} \label{lma-tQ}
Let the sequences $\{\G_t\}$ and $\{\mbf{q}^{\ell}_t \},$ $\ell = (i,u,v) \in \Stc$ be as defined by Eqs.~(\ref{eq-definitg}), (\ref{eq-definitp}), respectively.  Consider any sample path from the set of probability one given in Lemma~\ref{lma-init-tQ}.
For any $\delta > 0$, with $t_0$ and $\tilde Q_{t_0}$ given as in Section~\ref{sec-bl-init} for the chosen $\delta$, the sequence $\{\tilde Q_t\}$ defined by Eqs.~(\ref{eq-tQ-0})-(\ref{eq-tQ}) has the following properties. For each $\ell=(i,u,v) \in \Stc$ and $t \geq 0$:\moveup
\begin{enumerate}
\item[\rm (a)] 
$\tilde Q_{t} (i,u,v)$ can be expressed as
$$\tilde Q_{t} (i,u,v) = g^{\ell,t}_0(i_0, u_0, v_0) + \E^{ \mbf{P}^{\ell}_{t}} \Big[ \, \sum_{k=1}^\infty  g^{\ell,t}_k(i_k, u_k) \, 
\Big] = g^{\ell,t}_0(i_0, u_0, v_0)  + \sum_{k=1}^\infty  \E^{ \mbf{P}^{\ell}_{t}} \Big[ \,  g^{\ell,t}_k(i_k, u_k) \, 
\Big]$$
for some probability distribution $\mbf{P}^{\ell}_t$ of a time-inhomogeneous Markov chain $(i_0, u_0, v_0)$, $(i_1, u_1)$, $(i_2, u_2), \ldots$ with $(i_0, u_0,v_0) = (i,u,v)$ and $(i_k,u_k) \in \Ste \times \U$ for $k \geq 1$,
and for some one-stage cost functions $g^{\ell,t}_0: \Stc \mapsto \Re$, $g_k^{\ell,t}: \Stce^1 \mapsto \Re$, $k \geq 1$ (with $g_k^{\ell,t} \equiv +\infty$ on $(\Ste \times \U) \setminus \Stce^1$). 
\item[\rm (b)]  The transition probabilities of the Markov chain in (a) are time-varying and have the following product form: for all $(\bar i,\bar u), (\bar j, \bar w) \in \Stce^1$,
\begin{align*}
\mbf{P}^{\ell}_t \big( i_1 = \bar j, u_1 = \bar w \mid i_{0} = i, u_{0} = u, v_0 = v \big) & =   p^{\ell,t}_{0}(\bar j \mid i, u, v) \cdot  \mu^{\ell,t}_1 (\bar w \mid \bar j),  \qquad \text{for} \ k  = 1, \\
   \mbf{P}^{\ell}_t \big( i_k = \bar j, u_k = \bar w \mid i_{k-1} = \bar i, u_{k-1}  = \bar u \big) & =   p^{\ell,t}_{k-1}(\bar j \mid \bar i, \bar u) \cdot  \mu^{\ell,t}_k ( \bar w \mid \bar j),   \qquad \ \text{for} \ k \geq 2,
\end{align*}   
where $p^{\ell,t}_{k}$ and $\mu^{\ell,t}_k$ are conditional probability distributions such that for all $k \geq 1$ and  $(\bar i, \bar u) \in \Stce^1$, $\bar j \in \Ste$, 
$$  
p^{\ell,t}_{k}(\cdot \mid \bar i, \bar u)   \in \A_\delta \big(\mbf{p}_{\bplb}^{\bar i \bar u} \big), \qquad \mu^{\ell,t}_k ( \cdot \mid \bar j) \in \mathcal{P}(\U) \ \text{with}  \    \supp\big(\mu^{\ell,t}_k ( \cdot \mid \bar j)  \big) \subset U(\bar j),$$  
and for $k = 0$, $p^{\ell,t}_0(\cdot \mid i, u, v) \in \A_\delta \big(\mbf{p}^\ell_o\big)$ with $p^{\ell,t}_0(\cdot \mid i, u, v) = \mbf{q}^{\ell}_t$ when $t \geq t_0$.
\item[\rm (c)] The one-stage cost functions $g^{\ell,t}_k$ in (a) satisfy
\begin{equation}
  g^{\ell,t}_0 \in \B_\delta(g), \qquad g^{\ell,t}_k \in \B_\delta(g_{\bplb}), \qquad k \geq 1, \notag
\end{equation}  
with $g^{\ell,t}_0(i,u,v) = \G_t(i,u,v)$ for $t \geq t_0$.
\item[\rm (d)] For the Markov chain in (a), there exists an integer $k_t \geq 1$ such that $\{ (i_k, u_k), k \geq k_t \}$ evolves and incurs costs as in \sspa\ under the proper policy $\mu$; i.e., for $k \geq k_t$,
$$  \mu^{\ell,t}_{k}(\cdot \mid \bar i) = \mu( \cdot \mid \bar i), \qquad p^{\ell,t}_{k}(\cdot \mid \bar i, \bar u) = \mbf{p}_{\bplb}^{\bar i \bar u}, \qquad g^{\ell,t}_k(\bar i, \bar u) = g_{\bplb}(\bar i, \bar u), \quad \forall \, (\bar i, \bar u) \in \Stce^1.$$
\end{enumerate}
\end{lemma}
\smallskip

The proof of Lemma~\ref{lma-tQ} is by induction on $t$ for each $(i,u,v) \in \Stc$. In the proof, we construct the one-stage cost functions and transition probabilities for the time-inhomogeneous Markov chain associated with $t+1$, and this procedure resembles the construction of a cost-equivalent Markov policy in the classical MDP theory, for any given history-dependent policy and any given initial state. Other than the construction, the proof of Lemma~\ref{lma-tQ} consists of mostly straightforward verifications of the properties (a)-(d) in the statement. Nevertheless, the verifications turn out to be lengthy, so we give the proof of Lemma~\ref{lma-tQ} in~\ref{app-sec1}.

\subsubsection{Lower boundedness of $\{\tilde Q_t\}$}\label{sec-4.2.6}

We now come to the final step of our boundedness analysis: to lower-bound the optimal total costs of those time-inhomogeneous SSP problems neighboring \sspa\ and thereby lower-bound $\{\tilde Q_t\}$.
As we have shown with Lemmas~\ref{lma-init-tQ} and \ref{lma-tQ},  for each sample path from a set of probability one, and for each $\delta > 0$, we can construct a sequence $\{\tilde Q_t\}$ such that $\tilde Q_{t}(i,u,v)$ for each $(i,u,v) \in \Stc$ and $t \geq 0$ is the expected total cost of a randomized Markov policy in an SSP problem that has time-varying transition and one-stage cost parameters lying in the $\delta$-neighborhoods of the respective parameters of \sspa. 
As we show below,
when $\delta$ is sufficiently small, the total costs in all such neighboring SSP problems can be bounded uniformly from below.

Let us be precise about the type of SSP problems involved here. Consider all time-inhomogenous SSP problems that have the same state-control space as \sspa\ and have initial states in $\Stc$ [cf.\ the definition of \sspa\ given in Definition~\ref{def-sspa}]. For such an SSP, denote by $p_k$ and $g_k$ its state transition probability and its one-stage cost function, respectively, at the $k$th stage.  Let us call such an SSP  \emph{a $\delta$-perturbed version of \sspa} if 
for $k = 0$,
$$ g_0 \in \B_\delta(g),  \qquad p_0(\cdot \mid i,u,v) \in \A_\delta(\mbf{p}^\ell_o), \quad \forall \, \ell=(i,u,v) \in \Stc,  $$
and for $k \geq 1$,
$$ g_k \in \B_\delta(g_{\bplb}), \qquad p_k(\cdot \mid i, u) \in \A_\delta(\mbf{p}^{iu}_{\bplb}), \quad \forall \, (i,u) \in \Stce^1.$$

Because \sspa\ satisfies the SSP Model Assumption (Lemma~\ref{lma-sspa}), we have the following lemma. It was proved in Yu and Bertsekas~\cite[Section 3.3.4]{yub_bd11} and we will not repeat the proof here. The idea of the proof is to consider a time-homogeneous compact-control SSP problem where the controls include not only the regular controls but also the transition/one-stage cost parameters in the $\delta$-neighborhoods $\A_{\delta}(\mbf{p}_o^{\ell})$, $\A_{\delta}(\mbf{p}_{\bplb}^{iu})$, $\B_{\delta}(g)$, $\B_{\delta}(g_{\bplb})$ of the respective parameters of \sspa, and to show that the optimal total costs of this compact-control SSP are finite when $\delta$ is sufficiently small, by using a continuity argument together with the fact that \sspa\ satisfies the SSP model assumption and hence has finite optimal total costs by \cite{bet-ssp91}.

\smallskip
\begin{lemma}[{\cite[Section 3.3.4]{yub_bd11}}] \label{lma-bdperturbssp}
There exist  $\bar \delta > 0$ and a finite constant $C$ such that for all $\delta \in [0, \bar \delta]$, the optimal total cost of any $\delta$-perturbed version of \sspa, for any initial state, is greater than $C$.
\end{lemma}

Combining Lemma~\ref{lma-tQ} with Lemma~\ref{lma-bdperturbssp}, we obtain the boundedness of $\{\tilde Q_t\}$ as stated below.

\smallskip
\begin{lemma} \label{lma-tQ-bdbl}
Let $\delta \in (0, \bar \delta]$ where $\bar \delta$ is as given in Lemma~\ref{lma-bdperturbssp}. 
Then on any sample path from the set of probability one given in Lemma~\ref{lma-init-tQ}, 
with $t_0$ and $\tilde Q_0$ defined as in Section~\ref{sec-bl-init} for the chosen $\delta$, the sequence $\{\tilde Q_{t} \}$ defined by Eqs.~(\ref{eq-tQ-0})-(\ref{eq-tQ}) is bounded below.
\end{lemma}
\smallskip

Lemma~\ref{lma-tQ-bdbl} together with Lemma~\ref{lma-bdbl-hQ} implies that the sequence $\{\hat Q_t\}$ generated by the iteration~(\ref{eq-bdbl-Q}) is bounded below w.p.$1$, which in turn implies, by Lemma~\ref{lma-bdbl-Q}, that the Q-learning iterates $\{Q_t\}$ are bounded below w.p.$1$. 
A symmetric argument then yields that $\{Q_t\}$ is bounded above w.p.$1$, as we explained at the beginning of Section~\ref{sec-bd-gen}. 
This completes the proof of Theorem~\ref{thm-boundql} on the boundedness of Q-learning iterates $\{Q_t\}$ for SSP games satisfying Assumption~\ref{assum-ssp}.

\subsection{Boundedness Analysis for a Special Case} \label{sec-bd-special}

If instead of Assumption~\ref{assum-ssp}, we make a stronger model assumption on the SSP game, 
then there is a simpler proof of the boundedness (and hence convergence) of Q-learning iterates, based on a contraction argument.
We present this analysis to conclude Section~\ref{sec-bd}.

First, let us define a notion of proper policies for an SSP game and use it to formulate the stronger model assumption just mentioned. (Our definition of properness differs slightly from that in \cite{sspgame_pt99}.) Recall that a pair of policies of the two players is non-prolonging if under those policies, the termination state is reached w.p.$1$ for all initial states.

\smallskip
\begin{definition}[Proper Policies in a Finite-Space SSP Game] \label{def-proper-sspgame}
For a finite-space SSP game, we say a policy $\pla \in \Ra$ of player I is \emph{proper} if for \emph{every} policy $\plb \in \Rb$ of player II, $(\pla,\plb)$ is non-prolonging. Similarly, we say a policy $\plb \in \Rb$ of player II is \emph{proper} if for \emph{every} policy $\pla \in \Ra$ of player I, $(\pla,\plb)$ is non-prolonging.
\end{definition}

\begin{assumption} \label{assump-4.1}
The following holds in a finite-space SSP game: \moveup\moveup
\begin{itemize}
\item[(i)] Every player has a proper stationary randomized policy.\moveup
\item[(ii)] For any pair of policies $(\npla, \nplb) \in \Ra \times \Rb$ that is prolonging, $J(i; \npla, \nplb) = + \infty$ or $- \infty$ for at least one initial state $i$.\moveup
\end{itemize}
\end{assumption}
\smallskip

By Definition~\ref{def-proper-sspgame}, if a player plays a proper policy, the optimal total costs or rewards for the other player are finite for all initial states. Therefore, an SSP game that satisfies the model conditions in Assumption~\ref{assump-4.1} also satisfies Assumption~\ref{assum-ssp}. Consequently, the optimality results of Section~\ref{sec-finitessp} and the boundedness and convergence theorems for Q-learning hold under Assumption~\ref{assump-4.1} as well.
However, for proper policies of either players, the dynamic programing operators of their associated single-player problems exhibit a contraction property. This gives us a shortcut to prove the boundedness of Q-learning iterates under Assumption~\ref{assump-4.1}(i), without resorting to the long, general-case proof given earlier.

\smallskip
\begin{proposition}[Boundedness of Q-Learning Iterates in the Presence of Proper Policies] \label{prp-bd-special} 
Suppose there exists a proper policy $\bpla \in \Ra$ of player I ($\bplb \in \Rb$ of player II, respectively) in a finite-space SSP game. Then under Assumption~\ref{assump-qlalg}(i)-(iii) and (v), for any given initial $Q_0$, the sequence $\{Q_t\}$ generated by the Q-learning algorithm (\ref{eq-Q}) is bounded above (below, respectively) w.p.$1$.
\end{proposition}
\smallskip

We will prove the lower-boundedness part of Prop.~\ref{prp-bd-special}. By symmetry, the upper-boundedness part of Prop.~\ref{prp-bd-special} follows from applying the same argument to the process $\{ - Q_t \}$.
We start with a mapping $F_{\bplb} : \Re^{|\Stc|} \to \Re^{|\Stc|}$ for a policy $\bplb \in \Rb$ defined by
\begin{equation} \label{eq-def-hatF}
  (F_{\bplb} Q)(i,u,v) := g(i,u,v) +  \sum_{j \in \St} p_{ij}(u,v)  \inf_{\tilde u \in U(j)}  \ave Q \big(j,\tilde u, \bplb_j  \big), \qquad \forall \, (i,u,v) \in \Stc, \ \forall \, Q,
\end{equation}
where 
$$\bplb_j = \bplb(\cdot \mid j) \qquad \text{and} \qquad \ave Q \big(j, \tilde u, \bplb_j  \big) = \sum_{\tilde v \in V(j)} \bplb_j(\tilde v) \, Q (j, \tilde u, \tilde v).$$
Given a positive vector $\xi = \{ \xi(i,u,v) \mid (i,u,v) \in \Stc\}$, let  $\| \cdot \|_\xi$ denote the weighted sup-norm on the space of Q-factors given by $\| Q \|_\xi = \max_{(i,u,v) \in \Stc} \frac{ | Q(i,u,v) |}{\xi(i,u,v)}$.

\smallskip
\begin{lemma} \label{lma-contraction}
For a proper policy $\bplb \in \Rb$, $F_{\bplb}$ given by Eq.~(\ref{eq-def-hatF}) is a contraction with respect to some weighted sup-norm $\|\cdot\|_\xi$, i.e., for some $\beta \in [0,1)$,
$$  \| F_{\bplb} Q  - F_{\bplb} Q'  \|_\xi \leq \beta \, \| Q - Q' \|_\xi, \qquad \forall \, Q, Q'.$$
\end{lemma}

\begin{proof}
First, we define a single-player SSP problem and use its optimal total costs to construct the weight vector $\xi$ in the desired norm $\|\cdot\|_\xi$, similar to the proofs of \cite[Prop.~2.2, p.\ 23-24]{BET} and \cite[Lemma~4.1]{sspgame_pt99}.
Consider a single-player SSP problem on the state space $\Ste \cup \Stc$ where the system dynamics are the same as those of \sspa\ given in Definition~\ref{def-sspa}, and all the one-stage costs are $-1$ except for that at the cost-free termination state $0$. 
Because $\bplb$ is a proper policy of player II, by the definition of a proper policy in an SSP game (cf.\ Definition~\ref{def-proper-sspgame}),
the single-player SSP problem just defined satisfies the SSP Model Assumption (cf.\ Section~\ref{sec-sspgame-model}). 
Therefore, by \cite{bet-ssp91}, its optimal total cost function $\hat J^*$ is finite at all states in $\Ste \cup \Stc$ and satisfies the dynamic programming equation
\begin{align}
 \hat J^*(i,u,v) & = -1 + \sum_{j \in \St} p_{ij}(u,v)\, \hat J^*(j),  \qquad \forall \, (i, u, v)  \in \Stc,  \label{eq-prf-bdabv1} \\
  \hat J^*(i) & = -1 + \min_{u \in U(i)} \sum_{j \in \St} p_{\bplb,ij}(u) \, \hat J^*(j),  \qquad \forall \, i \in \St.  \label{eq-prf-bdabv2}
\end{align}
We also have that $\hat J^*(i) \leq -1$ and $\hat J^*(i,u,v) \leq -1$ for every state $i$ and $(i, u,v)$, since the one-stage costs before termination are $-1$.

Let us define 
\begin{align*}
    \xi(i,u,v)  & =  - \hat J^*(i, u,v) \geq 1, \qquad \quad \ \ (i,u,v) \in \Stc, \\
    \xi_{\bplb}(i, u) & =  \sum_{v \in V(i)} \bplb_i(v) \, \xi(i,u,v), \qquad i \in \St, \ u \in U(i).
\end{align*}    
For every $i \in \St$ and $u \in U(i)$,
by Eqs.~(\ref{eq-prf-bdabv1})-(\ref{eq-prf-bdabv2}) and the definition of $p_{\bplb,ij}$ [cf. Definition~\ref{def-sspa}(2)],
\begin{align*}
 \xi_{\bplb}(i, u) & =  - \Big( -1 +  \sum_{j \in \St} p_{\bplb,ij}(u) \, \hat J^*(j) \Big)  \\
     & \leq - \Big( -1 + \min_{\tilde u \in U(i)} \sum_{j \in \St} p_{\bplb,ij}(  \tilde u) \, \hat J^*(j) \Big) 
     = - \hat J^*(i),
\end{align*}
and hence
\begin{equation} 
    \sup_{u \in U(i)}   \xi_{\bplb}(i,  u) \leq - \hat J^*(i), \qquad \forall \, i \in \St.  \notag
\end{equation}    
Then with $\beta = \max_{(i, u, v) \in \Stc} \frac{\xi(i,u,v) - 1}{\xi(i,u,v)} \in [0,1)$. 
we have for every $(i,u,v) \in \Stc$,
\begin{equation} \label{eq-prf-bdabv3}
     \sum_{j \in \St} p_{ij}(u,v)  \sup_{\tilde u \in U(j)} \xi_{\bplb}(j,  \tilde u)   \leq  \sum_{j \in \St} p_{ij}(u,v) \big( - \hat J^*(j) \big) 
         =  \xi(i,u,v) - 1 
         \leq  \beta \, \xi(i,u,v),
\end{equation}
where the equality follows from Eq.~(\ref{eq-prf-bdabv1}).

We now prove that $F_{\bplb}$ is a contraction with respect to $\|\cdot\|_\xi$ and with modulus $\beta$.
By the definition of $F_{\bplb}$ [cf.\ Eq.~(\ref{eq-def-hatF})], for every $(i,u,v) \in \Stc$,
$$\big| (F_{\bplb} Q)(i,u,v) - (F_{\bplb} Q')(i,u,v) \big| =  \Big| \sum_{j \in \St} p_{ij}(u,v) \Big( \inf_{\tilde u \in U(j)} Q_{\bplb}(j, \tilde u) -   \inf_{\tilde u \in U(j)} Q'_{\bplb}(j, \tilde u) \Big) \Big|,$$
where we define $Q_{\bplb}(j, \tilde u) = \sum_{\tilde v \in V(j)} \bplb_j(\tilde v) \, Q(j, \tilde u, \tilde v)$ and we define $Q'_{\bplb}(j, \tilde u)$ similarly.
Let $\Delta = | Q - Q' |$ (the absolute values are taken component-wise). Using the preceding equation, we have for every $(i,u,v) \in \Stc$,
\begin{align*}
\big| (F_{\bplb} Q)(i,u,v) - (F_{\bplb} Q')(i,u,v) \big| 
& \leq \sum_{j \in \St} p_{ij}(u,v) \sup_{\tilde u \in U(j)} \sum_{\tilde v \in V(j)} \bplb_j(\tilde v) \, \Delta(j, \tilde u, \tilde v) \\  
 & =  \sum_{j \in \St} p_{ij}(u,v) \sup_{\tilde u \in U(j)} \sum_{\tilde v \in V(j)} \bplb_j(\tilde v)\, \xi(j, \tilde u, \tilde v) \cdot  \frac{\Delta(j, \tilde u, \tilde v)}{ \xi(j, \tilde u, \tilde v)} \\ 
 & \leq  \| Q - Q' \|_\xi \cdot \sum_{j \in \St} p_{ij}(u,v) \sup_{\tilde u \in U(j)} \xi_{\bplb}(j,  \tilde u) \\
 & \leq  \| Q - Q' \|_\xi \cdot \beta \, \xi(i,u,v),
\end{align*}     
where the last inequality follows from Eq.~(\ref{eq-prf-bdabv3}). This implies
$\| F_{\bplb} Q - F_{\bplb} Q' \|_\xi \leq \beta \| Q - Q' \|_\xi$.
\end{proof}
\smallskip

\begin{proof}[Proof of Prop.~\ref{prp-bd-special}]
We prove the lower-boundedness part of the proposition; as mentioned earlier, the upper-boundedness part follows from applying the same argument to the process $\{ - Q_t \}$. 

Consider the process $\{{\hat Q}_t\}$ defined by the iteration~(\ref{eq-bdbl-Q}) with $\bplb$ being a proper policy of player~II. 
By Lemma~\ref{lma-bdbl-Q}, to prove that the sequence $\{Q_t\}$ of Q-learning iterates is bounded below w.p.$1$, it is sufficient to prove that $\{ {\hat Q}_t\}$ is bounded below w.p.$1$. 
Now the iteration~(\ref{eq-bdbl-Q}) for $\{\hat Q_t\}$ can be equivalently written as: for every $\ell = (i,u,v) \in \Stc$ and $t \geq 0$,
$$  {\hat Q}_{t+1}(i,u,v) = ( 1 - \g_{t, \ell} ) {\hat Q}_t(i,u,v) + \g_{t, \ell} \big( F_{\bplb} {\hat Q}^{(\ell)}_t \big)(i,u,v) + \g_{t, \ell} \, w_{t, \ell}, $$
where $F_{\bplb}$ is the mapping given by~(\ref{eq-def-hatF}) and $w_{t, \ell}$ is a noise term given by
$$ w_{t, \ell} = \hat g(i, u, v, s) +  \inf_{\da \in \bar U(s)}  {\ave {\hat Q}}^{(\ell)}_t (s, \da, \bplb_s ) - \big( F_{\bplb}  {\hat Q}^{(\ell)}_t \big)(i,u,v) \qquad \text{with} \   s = j_t^\ell. $$
By Lemma~\ref{lma-contraction}, $F_{\bplb}$ is a contraction with respect to a weighted sup-norm, so we can apply the result of \cite{tsi94} for asynchronous stochastic approximation algorithms involving contraction mappings.
Direct calculation shows that for any given $\hat Q_0$, under Assumption~(\ref{assump-qlalg})(i)-(iii),  the noise terms, $w_{t, \ell}, \ell \in \Stc, t \geq 0$, satisfy the conditional mean and variance conditions required in the analysis of \cite{tsi94}:
$\E \big[ w_{t,\ell} \mid \mathcal{F}_t \big]  = 0$ w.p.$1$,
and  $\E \big[ w_{t,\ell}^2 \mid \mathcal{F}_t \big]  \leq A + B \max_{\ell' \in \Stc} \max_{\tau \leq t}  | \hat Q_{\tau}(\ell') |^2$ w.p.$1$, for some deterministic constants $A$ and $B$. Therefore, by \cite[Theorem 1]{tsi94}, for any given $\hat Q_0 = Q_0$, $\{{\hat Q}_t\}$ is bounded  w.p.$1$ under Assumption~(\ref{assump-qlalg})(i)-(iii) and (v). [Assumption~(\ref{assump-qlalg})(iv) is not needed for bounding the iterates, although it is needed for establishing their convergence.]
As mentioned earlier, by Lemma~\ref{lma-bdbl-Q}, this implies that for any given initial $Q_0$, $\{Q_t\}$ is bounded below w.p.$1$.
\end{proof}

\section*{Acknowledgements}
I thank Prof.\ Dimitri Bertsekas and Prof.\ John Tsitsiklis for helpful comments. 
This research was supported by the Air Force Grant FA9550-10-1-0412.

\addcontentsline{toc}{section}{References} 
\bibliographystyle{amsalpha} 
\let\oldbibliography\thebibliography
\renewcommand{\thebibliography}[1]{%
  \oldbibliography{#1}%
  \setlength{\itemsep}{0pt}%
}
{\fontsize{9}{11} \selectfont
\bibliography{qlearning_sspgame_bib}}


\appendix

\newcommand{\appsection}[1]{\let\oldthesection\thesection
  \renewcommand{\thesection}{Appendix \oldthesection}
  \section{#1}\let\thesection\oldthesection}

\newpage

\newcommand{\nocontentsline}[3]{}
\newcommand{\tocless}[2]{\bgroup\let\addcontentsline=\nocontentsline#1{#2}\egroup}

\addcontentsline{toc}{section}{Appendix A \ Proof of Lemma~\ref{lma-tQ}} 
\tocless\appsection{Proof of Lemma~\ref{lma-tQ}} \label{app-sec1}

The proof is by induction on $t$. For $t = t_0$, $\tilde Q_{t_0}$ satisfies the properties (a)-(d) in the lemma by its definition and our choice of the sample path and $t_0$ [cf.\ Lemma~\ref{lma-init-tQ} and Eqs.~(\ref{eq-Pt0-a})-(\ref{eq-tQ0})].
Since $\tilde Q_t = \tilde Q_{t_0}$ for $t < t_0$, they also satisfy properties (a)-(d).
So consider $t  \geq t_0$ and suppose these properties are satisfied by all $\tilde Q_\tau$, $0 \leq \tau \leq t$. Let us show that they are satisfied by $\tilde Q_{t+1}$.

Consider $\tilde Q_{t+1}(i,u,v)$ for each $\ell=(i,u,v) \in \Stc$. To simplify notation, denote $\gamma = \gamma_{t,\ell} \in [0,1]$ (cf.\ Lemma~\ref{lma-init-tQ}).
By Eq.~(\ref{eq-tQ}),
\begin{equation} \label{eq-prf-tQ}
 \tilde Q_{t+1}(i,u,v) =  ( 1 - \gamma )\, \tilde Q_t(i,u,v) +  \gamma \, \Big( \hat g(i, u,v,s) +   {\ave {\tilde Q}}^{(\ell)}_t(s, \tilde u , \bplb_s) \Big),
\end{equation} 
 where $s = j^{\ell}_t, \tilde u = u^{\ell}_t$,  ${\ave {\tilde Q}}^{(\ell)}_t(s, \tilde u , \bplb_s) = \sum_{\tilde v \in V(s)}  \bplb_s(\tilde v) \, {\tilde Q}^{(\ell)}_t(s, \tilde u, \tilde v)$, and
 $$  {\tilde Q}^{(\ell)}_t(s, \tilde u, \tilde v) =  {\tilde Q}_{\tau_{\ell \ell_{\tilde v}}(t)}(s, \tilde u, \tilde v) \quad \text{with} \   \ell_{\tilde v} = (s, \tilde u, \tilde v), \  \tau_{\ell \ell_{\tilde v}}(t) \leq t, \quad
 \forall \,  \tilde v \in V(s). $$
Let us use the simplified notation $\tau_{\tilde v} = \tau_{\ell \ell_{\tilde v}}(t)$ for $\tilde v \in V(s)$. By the induction hypothesis, we can express $\tilde Q_t(i,u,v)$ and each term ${\tilde Q}_{\tau_{\ell \ell_{\tilde v}}(t)}(s, \tilde u, \tilde v), \tilde v \in V(s)$, for $s \not = 0$, in the form given in the statement (a) of the lemma. 
Thus when $s \not= 0$, we can write Eq.~(\ref{eq-prf-tQ}) as
\begin{align}
\tilde Q_{t+1}(i,u,v) & =  \,  ( 1 - \gamma) \, g^{\ell,t}_0(i,u,v) + ( 1 - \gamma) \, \sum_{k=1}^\infty  \E^{ \mbf{P}^{\ell}_{t}} \Big[ \,  g^{\ell,t}_k(i_k, u_k) \, 
\Big]  \notag \\
 & \quad \ + \gamma \,  \hat g(i, u, v, s)  + \gamma \, \sum_{\tilde v \in V(s)} \bplb_s(\tilde v) \cdot \left( g^{\ell_{\tilde v},\tau_{\tilde v}}_0(s, \tilde u, \tilde v)  +  \sum_{k=1}^\infty  \E^{ \mbf{P}^{\ell_{\tilde v}}_{\tau_{\tilde v}}} \Big[ \,  g^{\ell_{\tilde v},\tau_{\tilde v}}_k(i_k, u_k) \,  
 \Big]  \right) \notag \\
 & =  \ \sum_{k = 0}^\infty C_k \label{eq-tQexp0}
 \end{align}
 where
 \begin{align}
  C_0 & =  ( 1 - \gamma )  \,  g^{\ell,t}_0(i, u, v)  +  \gamma \,   \hat g(i, u, v, s),   \\
  C_1 & =  ( 1 - \gamma ) \,   \E^{ \mbf{P}^{\ell}_{t}} \Big[ \,  g^{\ell,t}_1(i_1, u_1) \, \Big]   
  + \gamma \,   \sum_{\tilde v \in V(s)} \bplb_s(\tilde v) \cdot g^{\ell_{\tilde v},\tau_{\tilde v}}_0(s, \tilde u, \tilde v),   \label{eq-tQexpc1} \\
  C_k & = ( 1 - \gamma ) \,   \E^{ \mbf{P}^{\ell}_{t}} \Big[ \,  g^{\ell,t}_k(i_k, u_k) \, \Big]   
 + \gamma \,   \sum_{\tilde v \in V(s)} \bplb_s(\tilde v) \cdot  \E^{ \mbf{P}^{\ell_{\tilde v}}_{\tau_{\tilde v}}} \Big[ \,  g^{\ell_{\tilde v}, \tau_{\tilde v}}_{k-1}(i_{k-1}, u_{k-1}) \,  
 \Big], \quad k \geq 2. \label{eq-tQexp2}
\end{align} 

For the sake of convenience, let us define a few terms for the case $s=0$ and make the above formulas valid for $s = 0$ as well. 
Recall that for $s=0$, we have $U(0)=V(0)=\{0\}$ and with $\tilde v = 0$, $\ell_{\tilde v} = (0,0,0)$ and ${\tilde Q}_{\tau}(s, \tilde u, \tilde v) = \tilde Q_\tau(0,0,0) = 0$ for all $\tau$. Let us set $\tau_{\tilde v} = 0$ in this case (since this term can be defined arbitrarily). To express $0$ in the form given in the lemma, 
let us simply define $\mbf{P}^{\ell_{\tilde v}}_{\tau_{\tilde v}}$ in this case to be the probability distribution of the Markov chain $\{(i_k,u_k), k \geq 0\}$ that starts from the absorbing termination state $(i_0, u_0) = (0,0)$; let $g^{\ell_{\tilde v},\tau_{\tilde v}}_0(0,0,0) = 0$ and let $g^{\ell_{\tilde v},\tau_{\tilde v}}_0$ coincide with $g$ elsewhere; and let $g^{\ell_{\tilde v},\tau_{\tilde v}}_k = g_{\bplb}$, $k \geq 1$. With these definitions, we have $0 = g^{\ell_{\tilde v},\tau_{\tilde v}}_0(s, \tilde u, \tilde v)  +  \sum_{k=1}^\infty  \E^{ \mbf{P}^{\ell_{\tilde v}}_{\tau_{\tilde v}}} \Big[ \,  g^{\ell_{\tilde v},\tau_{\tilde v}}_k(i_k, u_k) \, \Big]$ and  Eq.~(\ref{eq-tQexp0}) holds for $s = 0$. 
For later use, let us also define transition probabilities and other quantities so that some properties in the statement of the lemma hold for $s=0$. In particular, let $\mu_k^{\ell_{\tilde v}, \tau_{\tilde v}} = \mu$, $p_k^{\ell_{\tilde v}, \tau_{\tilde v}}(\cdot \mid \bar i, \bar u) = \mbf{p}^{\bar i \bar u}_{\bplb}$ for $k \geq 1$ and $(\bar i, \bar u) \in \Stce^1$, and also let $p_0^{\ell_{\tilde v}, \tau_{\tilde v}}(\cdot \mid 0,0,0) = \mbf{p}^{00}_{\bplb}$ (i.e., $p_0^{\ell_{\tilde v}, \tau_{\tilde v}}(0 \mid 0,0,0)=1$).
Then $\mbf{P}^{\ell_{\tilde v}}_{\tau_{\tilde v}}$ can be expressed in the product form given in property (b), and it satisfies property (d) with $k_{\tau_{\tilde v}} = 1$.

We now rewrite each term $C_k$ in the above expression of $\tilde Q_{t+1}(i,u,v)$ in a desirable form, first for $k = 0$, then for $k \geq 2$, and finally, for $k = 1$. 
During this procedure, we will define the transition probabilities $p^{\ell,t+1}_k$ and $\mu^{\ell,t+1}_k$ that compose the probability distribution $\mbf{P}^{\ell}_{t+1}$ of the time-inhomogenous Markov chain for $t+1$, as well as the one-stage cost functions $g^{\ell, t+1}_k$ required in the statement of the lemma.

\medskip
\noindent {\bf For $k=0$: }
By property (c) of the induction hypothesis, $g^{\ell,t}_0(i,u,v) = \G_t(i,u,v)$. Using this and the definition of $\{\G_t\}$ [cf.\ Eq.~(\ref{eq-definitg})], we have that 
\begin{equation} \label{eq-prftQ-1}
  C_0 =  ( 1 - \gamma )  \,  \G_t(i, u, v)  +  \gamma \,   \hat g(i, u, v, s)  = \G_{t+1}(i,u,v).
\end{equation} 
Let the cost function $g^{\ell,t+1}_0$ and transition probability $p^{\ell,t+1}_0(\cdot \mid i, u,v)$ be
\begin{equation} \label{eq-prf-def-gp0}
g^{\ell,t+1}_0 = \G_{t+1}, \qquad p^{\ell,t+1}_0(\cdot \mid i,u,v) = \mbf{q}_{t+1}^{\ell}.
\end{equation}
By Lemma~\ref{lma-init-tQ} and our choice of the sample path, $g^{\ell,t+1}_0$ and $p^{\ell,t+1}_0$ satisfy the requirements in properties (b) and (c), that is,
$$ g^{\ell,t+1}_0 \in \B_\delta (g), \qquad   p^{\ell,t+1}_0(\cdot \mid i, u,v)  \in \A_\delta\big(\mbf{p}^{\ell}_o \big).  $$

\medskip
\noindent {\bf For $k \geq 2$: }
Let $P^k_1$ denote the law of $(i_k, u_k, i_{k+1})$ under $\mbf{P}^{\ell}_{t}$, and for each $\tilde v \in V(s)$, let $P^{k,\tilde v}_2$ denote the law of $(i_{k-1}, u_{k-1}, i_{k})$ under  $\mbf{P}^{\ell_{\tilde v}}_{\tau_{\tilde v}}$. 
Let $P_3^k$ denote the convex combination of them, 
\begin{equation} \label{eq-def-pk3}
  P^k_3 = ( 1 - \gamma) P^k_1 + \gamma \, \sum_{\tilde v \in V(s)} \bplb_s(\tilde v) \cdot P^{k, \tilde v}_2.
\end{equation}  
We regard these laws as probability measures on the sample space $\tilde \Omega = \Ste \times \U \times \Ste$, and we denote by $X, Y$ and $Z$ the function that maps a point $\omega = (\bar i, \bar u, \bar j) \in \tilde \Omega$ to its $1$st, $2$nd and $3$rd coordinate, respectively. 
Using property (b) of  $\mbf{P}^{\ell}_{t}$ and $\mbf{P}^{\ell_{\tilde v}}_{\tau_{\tilde v}}$ from the induction hypothesis (in particular, using the property of $\{\mu_k^{\ell,t}, k \geq 1\}$, $\{\mu_k^{\ell_{\tilde v}, \tau_{\tilde v}}, k \geq 1 \}$),  it is clear that $\supp(P^k_3)  \subset \Stce^1 \times \Ste$, a subset of $\tilde \Omega$.
So we can write the term $C_k$ in Eq.~(\ref{eq-tQexp2}) for each $k \geq 2$ as
\begin{equation} 
   \sum_{\bar i \in \Ste} \sum_{  \bar u \in U(\bar i)} \, \Big( \, ( 1 - \gamma ) \, P^k_1(X = \bar i, Y = \bar u) \cdot  g^{\ell,t}_k( \bar i, \bar u) +  \gamma  \, \sum_{\tilde v \in V(s)} \bplb_s(\tilde v) \cdot P^{k, \tilde v}_2(X = \bar i, Y = \bar u) \cdot g^{\ell_{\tilde v}, \tau_{\tilde v}}_{k-1}( \bar i, \bar u) \, \Big). \notag
\end{equation}   

Next we will define the $k$th-stage cost function $g^{\ell, t+1}_k$ so that  we can rewrite the above expression of $C_k$ equivalently as
\begin{equation} \label{eq-prf-exptQk1}
   C_k = \sum_{\bar i \in \Ste} \sum_{  \bar u \in U(\bar i)}  \, P^k_3(X = \bar i, Y = \bar u) \,\cdot g^{\ell, t+1}_k(\bar i, \bar u).
\end{equation}  
We will also define the transition probabilities $\mu^{\ell, t+1}_k(\cdot \mid \bar i)$ and $p^{\ell,t+1}_k(\cdot \mid \bar i, \bar u)$ for all $(\bar i, \bar u) \in \Stce^1$ so that we have
for every $(\bar i, \bar u) \in \Stce^1$ and  $\bar j \in \Ste$, 
\begin{align}
P^k_3(X = \bar i, Y = \bar u) & = P^k_3(X = \bar i) \cdot \mu^{\ell, t+1}_k(\bar u \mid \bar i),  \label{eq-prf-nu} \\
P^k_3(X = \bar i, Y = \bar u, Z = \bar j ) & = P^k_3(X = \bar i, Y = \bar u) \cdot p^{\ell,t+1}_k(\bar j \mid\bar i, \bar u). \label{eq-prf-tp} 
\end{align} 

We define the cost function $g^{\ell, t+1}_k$ as follows. For each $(\bar i, \bar u) \in \Stce^1$, if $P^k_3(X = \bar i, Y = \bar u)  = 0$, let $g^{\ell, t+1}_k(\bar i, \bar u) = g_{\bplb}(\bar i, \bar u)$; otherwise, let
\begin{equation} \label{eq-defg}
 g^{\ell, t+1}_k(\bar i, \bar u) = \frac{( 1 - \gamma ) \, P^k_1(X = \bar i, Y = \bar u)}{P^k_3(X = \bar i, Y = \bar u)} \cdot g^{\ell,t}_k( \bar i, \bar u) +   \sum_{\tilde v \in V(s)} \frac{ \gamma  \, \bplb_s(\tilde v) \, P^{k, \tilde v}_2(X = \bar i, Y = \bar u)}{P^k_3(X = \bar i, Y = \bar u)}  \cdot g^{\ell_{\tilde v}, \tau_{\tilde v}}_{k-1}( \bar i, \bar u).
\end{equation}
Then, by the definition of $P^k_3$ [cf.\ Eq.~(\ref{eq-def-pk3})], Eq.~(\ref{eq-prf-exptQk1}) clearly holds.
Observe from Eq.~(\ref{eq-defg}) that $g^{\ell, t+1}_k(\bar i, \bar u)$ is a convex combination of $g^{\ell,t}_k( \bar i, \bar u)$ and $g^{\ell_{\tilde v}, \tau_{\tilde v}}_{k-1}( \bar i, \bar u), \tilde v \in V(s)$. The latter terms, by property (c) of the induction hypothesis and by the definitions we gave for the case $s=0$, all lie in the $\delta$-neighborhood of $g_{\bplb}(\bar i, \bar u)$, and they all equal $g_{\bplb}(\bar i, \bar u) = 0$ if $(\bar i, \bar u) = (0,0)$.  Hence, when $P^k_3(X = \bar i, Y = \bar u)  > 0$ and $g^{\ell, t+1}_k(\bar i, \bar u)$ is given by Eq.~(\ref{eq-defg}), $\big| g^{\ell, t+1}_k(\bar i, \bar u) - g_{\bplb}(\bar i, \bar u) \big| \leq \delta$, and $g^{\ell, t+1}_k(\bar i, \bar u) = 0 $ if $(\bar i, \bar u) = (0,0)$.  This shows that $g^{\ell, t+1}_k$ satisfies the requirement in property (c) for $t+1$:
$g^{\ell, t+1}_k \in \B_\delta(g_{\bplb}).$

Reasoning similarly, since by property (d) of the induction hypothesis, when $k \geq  k_t$ and $k \geq \max_{\tilde v \in V(s)} k_{\tau_{\tilde v}} + 1$, $g^{\ell,t}_k( \bar i, \bar u)= g^{\ell_{\tilde v}, \tau_{\tilde v}}_{k-1}( \bar i, \bar u) = g_{\bplb}(\bar i, \bar u)$ for all $\tilde v \in V(s)$ and $(\bar i, \bar u) \in \Stce^1$, it follows that $g^{\ell, t+1}_k$ satisfies the requirement in property (d) for $t+1$:
\begin{equation} 
  g^{\ell, t+1}_k = g_{\bplb}, \qquad \forall \, k \geq k_{t+1} := \max \Big\{ k_t,  \max_{\tilde v \in V(s) } k_{\tau_{\tilde v}} + 1 \Big\}. \notag
\end{equation}

Define the transition probability distributions $\mu^{\ell, t+1}_k$ and  $p^{\ell,t+1}_k$ by
\begin{align} 
     \mu^{\ell, t+1}_k(\cdot \mid \bar i) & = P^k_3(Y = \cdot \mid X = \bar i), \qquad
 \forall \, \bar i \in \Ste,  \label{eq-prf-deftrans1}\\
 p^{\ell,t+1}_k(\cdot \mid \bar i, \bar u) & = P^k_3(Z = \cdot \mid X = \bar i, Y = \bar u), \qquad \forall \, (\bar i, \bar u) \in \Stce^1.  \label{eq-prf-deftrans2}
\end{align} 
In the right-hand sides of Eqs.~(\ref{eq-prf-deftrans1})-(\ref{eq-prf-deftrans2}), in case an event that is conditioned on has probability zero, the corresponding conditional probability, which can be defined arbitrarily, is defined according to the proper policy $\mu$ or the transition probabilities of \sspa\ as:
 \begin{align*}
   P^k_3(Y = \cdot \mid X = \bar i) & = \mu( \cdot \mid \bar i),  &&  \text{if} \ \ P^k_3(X = \bar i) = 0;  \\
 P^k_3(Z = \cdot \mid X = \bar i, Y = \bar u) & = \mbf{p}^{\bar i\bar u}_{\bplb}, &&  \text{if} \ \ P^k_3(X = \bar i, Y = \bar u) = 0.
 \end{align*}
The desired equalities~(\ref{eq-prf-nu})-(\ref{eq-prf-tp}) then hold by these definitions. 
We now verify that $\mu^{\ell, t+1}_k$ and  $p^{\ell,t+1}_k$  satisfy the requirements in properties (b) and (d) for $t+1$. 

First, we show that $p^{\ell,t+1}_{k}$ satisfies the requirement in property (b):
$$ p^{\ell,t+1}_k(\cdot \mid \bar i, \bar u) \in \A_\delta \big(\mbf{p}^{\bar i \bar u}_{\bplb} \big), \qquad \forall \, (\bar i, \bar u) \in \Stce^1.$$
This holds by definition if $P^k_3(X = \bar i, Y = \bar u) = 0$, so we consider the case $P^k_3(X = \bar i, Y = \bar u) > 0$. 
By the induction hypothesis and by the definitions we made for the case $s=0$, $\mbf{P}^{\ell}_t$ and $\mbf{P}^{\ell_{\tilde v}}_{\tau_{\tilde v}}$, $\tilde v \in V(s)$, all have the product form given in property (b). 
Using the definition of $P^k_1$ and $P^{k, \tilde v}_2$, we then have that for all $\bar j \in \Ste$,
\begin{align*}
 P^k_1(X = \bar i, Y = \bar u, Z = \bar j)  & =  \mbf{P}^{\ell}_{t}\big( i_k = \bar i , u_k = \bar u 
 \big) \cdot 
 p^{\ell,t}_k(\bar j \mid\bar i, \bar u), \\
 P^{k, \tilde v}_2(X = \bar i, Y = \bar u, Z = \bar j)  & = \mbf{P}^{\ell_{\tilde v}}_{\tau_{\tilde v}}\big( i_{k-1} = \bar i , u_{k-1} = \bar u 
 \big) \cdot 
 p^{\ell_{\tilde v},\tau_{\tilde v}}_{k-1}(\bar j \mid\bar i, \bar u), \qquad  \tilde v \in V(s).
\end{align*}
This implies that for every $(\bar i, \bar u) \in \Stce^1$ and every $\tilde v \in V(s)$,
\begin{equation} \label{eq-prf-trans}
 P^k_1( Z = \cdot \mid X = \bar i, Y = \bar u) =  p^{\ell,t}_k(\cdot \mid \bar i, \bar u), \quad P^{k, \tilde v}_2( Z = \cdot \mid X = \bar i, Y = \bar u)  =  p^{\ell_{\tilde v},\tau_{\tilde v}}_{k-1}(\cdot \mid \bar i, \bar u).
\end{equation} 
Then, since $P^k_3 = ( 1 - \gamma) P^k_1 + \gamma \sum_{\tilde v \in V(s)} \bplb_s(\tilde v) P^{k, \tilde v}_2$,  using Eqs.~(\ref{eq-prf-deftrans2}), (\ref{eq-prf-trans}) and the relation $P^k_3( Z  = \cdot \mid X = \bar i , Y = \bar u)  = P^k_3( X = \bar i, Y = \bar u, Z = \cdot) / P^k_3(X = \bar i, Y = \bar u)$, we obtain
\begin{align}
  p^{\ell,t+1}_k(\cdot \mid \bar i, \bar u)   & = 
 \frac{ ( 1 - \gamma) P^k_1(X = \bar i, Y = \bar u)}{P^k_3(X = \bar i, Y = \bar u)}     \cdot p^{\ell,t}_k(\cdot \mid \bar i, \bar u)  \notag \\
 & \quad \ +  \sum_{\tilde v \in V(s)} \frac{ \gamma \, \bplb_s(\tilde v) \, P^{k, \tilde v}_2(X = \bar i, Y = \bar u)}{P^k_3(X = \bar i, Y = \bar u)}   \cdot p^{\ell_{\tilde v},\tau_{\tilde v}}_{k-1}(\cdot \mid \bar i, \bar u). 
   \label{eq-prf-defp}
\end{align}
This shows that $p^{\ell,t+1}_k(\cdot \mid \bar i, \bar u)$ is a convex combination of $p^{\ell,t}_k(\cdot \mid \bar i, \bar u)$ and $p^{\ell_{\tilde v},\tau_{\tilde v}}_{k-1}(\cdot \mid \bar i, \bar u)$, $\tilde v \in V(s)$.
By property (b) of the induction hypothesis,
\begin{equation} 
   p^{\ell,t}_k(\cdot \mid \bar i, \bar u) \in \A_\delta\big(\mbf{p}^{\bar i  \bar u}_{\bplb}\big), \qquad  p^{\ell_{\tilde v},\tau_{\tilde v}}_{k-1}(\cdot \mid \bar i, \bar u) \in \A_\delta\big(\mbf{p}^{\bar i \bar u}_{\bplb}\big), \quad \forall \, \tilde v \in V(s). \notag
\end{equation}   
Since the set $\A_\delta\big(\mbf{p}^{\bar i  \bar u}_{\bplb}\big)$ is convex,  
it follows that
$p^{\ell,t+1}_k(\cdot \mid \bar i, \bar u)  \in \A_\delta\big(\mbf{p}^{\bar i  \bar u}_{\bplb}\big)$, so it satisfies the requirement in property (b).

Reasoning similarly, 
and using property (d) of the induction hypothesis, 
it follows that for all $(\bar i, \bar u) \in \Stce^1$,
\begin{equation} 
 p^{\ell,t+1}_k(\cdot \mid \bar i, \bar u) =  \mbf{p}_{\bplb}^{\bar i \bar u}, \qquad \forall \, k \geq k_{t+1}. \notag
\end{equation} 
So $p^{\ell,t+1}_{k}$ satisfies the requirement in property (d) for $t+1$. 

We now verify that $\mu^{\ell, t+1}_k$ satisfies the requirements in properties (b) and (d) for $t+1$. 
Similar to the preceding proof, for each $\bar i \in \Ste$, either $\mu^{\ell, t+1}_k(\cdot \mid \bar i) = \mu(\cdot \mid \bar i)$ (when $P^k_3(X = \bar i) = 0$), 
or it
can be expressed as a convex combination of $\mu^{\ell,t}_k(\cdot \mid \bar i)$ and $\mu^{\ell_{\tilde v}, \tau_{\tilde v}}_{k-1}(\cdot \mid \bar i)$, $\tilde v \in V(s)$:
\begin{align*}
   \mu^{\ell,t+1}_k(\cdot \mid \bar i) 
 =  \frac{ ( 1 - \gamma) P^k_1(X = \bar i)}{P^k_3(X = \bar i)}     \cdot \mu^{\ell,t}_k(\cdot \mid \bar i) 
   +  \sum_{\tilde v \in V(s)} \frac{ \gamma  \, \bplb_s(\tilde v) \, P^{k, \tilde v}_2(X = \bar i)}{P^k_3(X = \bar i)}   \cdot \mu^{\ell_{\tilde v}, \tau_{\tilde v}}_{k-1}(\cdot \mid \bar i).
\end{align*}  
It then follows from properties (b) and (d) of the induction hypothesis that
$\supp\big( \mu^{\ell, t+1}_k(\cdot \mid \bar i)\big) \subset U(\bar i)$ for all $\bar i \in \Ste$, and  
$ \mu^{\ell,t+1}_k = \mu$ for $k \geq k_{t+1},$  which are the requirements in properties (b) and (d).

\bigskip
\noindent {\bf For $k=1$: }
The arguments in this case are similar to those for $k \geq 2$.
We start with the same definitions. Let $P^1_1$ denote the law of $(i_1, u_1, i_2)$ under $\mbf{P}^{\ell}_t$,
and for each $\tilde v \in V(s)$, let $P^{1, \tilde v}_2$ denote the law of $(i_0, u_0, i_1)$ under $\mbf{P}^{\ell_{\tilde v}}_{\tau_{\tilde v}}$. 
Let $P^1_3$ denote the convex combination of them, given by Eq.~(\ref{eq-def-pk3}). Define the random variables $X, Y$ and $Z$ on the sample space $\tilde \Omega = \Ste \times \U \times \Ste$ as in the preceding case of $k \geq 2$.
Let $I[\cdots]$ denote the indicator function which takes the value $1$ if the expression inside $[\cdots]$ is true and takes the value $0$ otherwise. 
Since for every $\tilde v \in V(s)$,
$$ P^{1, \tilde v}_2( X = \bar i, Y = \bar u) = \mbf{P}^{\ell_{\tilde v}}_{\tau_{\tilde v}} ( i_0 = \bar i, u_0 = \bar u) =  I \big[ \, \bar i = s, \bar u = \tilde u \, \big], \qquad (\bar i, \bar u) \in \Stce^1,$$
we have
\begin{equation}
  P^1_3 ( X = \bar i, Y = \bar u) = ( 1 - \gamma) \, P^1_1 ( X = \bar i, Y = \bar u) + \gamma \, I \big[ \, \bar i = s, \bar u = \tilde u \, \big], \qquad  (\bar i, \bar u) \in \Stce^1. \label{eq-prf-P13}
\end{equation}  

Notice that $\supp(P^1_3) \subset \Stce^1 \times \Ste$ because $(s, \tilde u) = (j^\ell_t, u^\ell_t) \in \Stce^1$ and $\supp(P^1_1) \subset \Stce^1 \times \Ste$ by property~(b) of the induction hypothesis (in particular, the property of $\mu^{\ell,t}_1$).
Hence we can write the term $C_1$ in Eq.~(\ref{eq-tQexpc1}) as
\begin{equation}
 C_1 =  \sum_{\bar i \in \Ste} \sum_{  \bar u \in U(\bar i)}  \, P^1_3(X = \bar i, Y = \bar u) \,\cdot g^{\ell, t+1}_1(\bar i, \bar u), \label{eq-prf-expC1}
\end{equation}
where $g^{\ell, t+1}_1(\bar i, \bar u)$ for every $(\bar i, \bar u)$ is defined as: 
if $P^1_3(X = \bar i, Y = \bar u) = 0$, then $g^{\ell, t+1}_1(\bar i, \bar u) = g_{\bplb}(\bar i, \bar u)$; otherwise,
\begin{equation} \label{eq-prf-defg1}
 g^{\ell, t+1}_1(\bar i, \bar u) = \frac{( 1 - \gamma ) \, P^1_1(X = \bar i, Y = \bar u)}{P^1_3(X = \bar i, Y = \bar u)} \cdot g^{\ell,t}_1( \bar i, \bar u) +  \frac{ \gamma  \,  I \big[ \, \bar i = s, \bar u = \tilde u \, \big]}{P^1_3(X = \bar i, Y = \bar u)}  \cdot  \sum_{\tilde v \in V(s)}   \bplb_{s}(\tilde v) \, g^{\ell_{\tilde v}, \tau_{\tilde v}}_{0}( s, \tilde u, \tilde v), 
\end{equation}
which, for $(\bar i, \bar u) \not= (s, \tilde u)$, is $g^{\ell, t+1}_1(\bar i, \bar u) = g^{\ell,t}_1( \bar i, \bar u)$.

We verify that $g^{\ell, t+1}_1$ satisfies the requirement in property (c) for $t+1$:
$g^{\ell, t+1}_1 \in  \B_\delta(g_{\bplb}).$ By the definition of  $\B_\delta(g_{\bplb})$, what we need to show is that for each $(\bar i, \bar u) \in \Stce^1$,
\begin{equation} \label{eq-prf-ineqg1}
  \big| g^{\ell, t+1}_1(\bar i, \bar u) -  g_{\bplb}(\bar i, \bar u) \big| \leq \delta,\qquad \text{and} \quad  g^{\ell, t+1}_1(\bar i, \bar u) = 0 \ \ \text{if} \ \bar i = 0.
\end{equation}  
From the definition of $g^{\ell,t+1}_1$ and the fact that $g^{\ell,t}_1 \in \B_\delta(g_{\bplb})$ [property (c) of the induction hypothesis], 
we see that Eq.~(\ref{eq-prf-ineqg1}) is obviously true for all $(\bar i, \bar u) \not=(s, \tilde u)$ and for the case where $g^{\ell, t+1}_1(\bar i, \bar u) = g_{\bplb}(\bar i, \bar u)$.
This leaves us only one case to consider: $(\bar i, \bar u) = (s, \tilde u)$ and  $g^{\ell,t+1}_1(s, \tilde u)$ is given by Eq.~(\ref{eq-prf-defg1}).

By Eq.~(\ref{eq-prf-defg1}), $g^{\ell, t+1}_1(s, \tilde u)$ is a convex combination of $g^{\ell,t}_1( s, \tilde u)$ and  $\sum_{\tilde v \in V(s)}   \bplb_{s}(\tilde v) \, g^{\ell_{\tilde v}, \tau_{\tilde v}}_{0}( s, \tilde u, \tilde v)$.
If $s=0$, then the latter two terms both equal $0$ by the induction hypothesis and by our definition of $g^{\ell_{\tilde v}, \tau_{\tilde v}}_{0}$ for $s = 0$, and consequently $g^{\ell, t+1}_1(0, 0) = 0$ as desired. Consider now the case $s \not= 0$. 
By property (c) of the induction hypothesis,
$$  g^{\ell_{\tilde v},\tau_{\tilde v}}_0 \in \B_\delta(g), \quad \forall \, \tilde v \in V(s).$$
Since $g_{\bplb}(s, \tilde u) = \sum_{\tilde v \in V(s)}   \bplb_{s}(\tilde v) g(s, \tilde u, \tilde v)$ [cf.\ Eq.~(\ref{eq-sspa-transcost}) in Definition~\ref{def-sspa} for \sspa],  
this implies that
$$ \Big| g_{\bplb}(s, \tilde u) - \sum_{\tilde v \in V(s)}   \bplb_{s}(\tilde v) \, g^{\ell_{\tilde v}, \tau_{\tilde v}}_{0}( s, \tilde u, \tilde v)  \Big| \leq \delta, \qquad \text{if} \ s \not= 0.$$
Combining the preceding relations with the induction hypothesis that $g^{\ell,t}_1 \in \B_\delta(g_{\bplb})$, we have
$$  \big| g^{\ell, t+1}_1(s, \tilde u) -  g_{\bplb}(s, \tilde u) \big| \leq \delta,\qquad \text{and} \quad  g^{\ell, t+1}_1(s, \tilde u) = 0 \ \ \text{if} \ s = 0,$$
which is Eq.~(\ref{eq-prf-ineqg1}) for $(\bar i, \bar u) = (s, \tilde u)$. 
This proves that 
$g^{\ell, t+1}_1 \in  \B_\delta(g_{\bplb})$, which is the requirement in property (c).

We define the transition probability distributions $\mu^{\ell, t+1}_1$, $p^{\ell,t+1}_1$ by Eqs.~(\ref{eq-prf-deftrans1}), (\ref{eq-prf-deftrans2}), respectively, for $k=1$, so that Eqs.~(\ref{eq-prf-nu})-(\ref{eq-prf-tp}) hold for $k=1$ as well. Evidently $\mu^{\ell, t+1}_1$ satisfies the requirement in property (b) for $t+1$, because $\supp(P^1_3) \subset \Stce^1 \times \Ste$ as discussed earlier. We now verify that $p^{\ell,t+1}_1$ satisfies the requirement in property (b) for $t+1$, namely,
\begin{equation} \label{eq-prf-p1-propb}
  p^{\ell,t+1}_1(\cdot \mid \bar i, \bar u) \in \A_\delta \big(\mbf{p}^{\bar i \bar u}_{\bplb} \big), \qquad \forall \, (\bar i, \bar u) \in \Stce^1.
\end{equation}  
Similar to the analysis given earlier for the case $k \geq 0$, we have that for every $(\bar i, \bar u) \in \Stce^1$,  
either $P^1_3(X = \bar i, Y = \bar u) = 0$ and $p^{\ell,t+1}_1(\cdot \mid \bar i, \bar u) = \mbf{p}^{\bar i \bar u}_{\bplb}$ by definition, or $P^1_3(X = \bar i, Y = \bar u) > 0$
and $p^{\ell,t+1}_1(\cdot \mid \bar i, \bar u)$ can be expressed as the convex combination
\begin{align}
 p^{\ell,t+1}_1(\cdot \mid \bar i, \bar u) & = \frac{( 1 - \gamma ) \, P^1_1(X = \bar i, Y = \bar u)}{P^1_3(X = \bar i, Y = \bar u)} \cdot p^{\ell,t}_1( \cdot \mid \bar i, \bar u) \notag \\
   & \quad \ + 
  \frac{ \gamma  \,  I \big[ \, \bar i = s, \bar u = \tilde u \, \big]}{P^1_3(X = \bar i, Y = \bar u)}  \cdot  \sum_{\tilde v \in V(s)}   \bplb_{s}(\tilde v) \, p^{\ell_{\tilde v}, \tau_{\tilde v}}_0(\cdot \mid s, \tilde u, \tilde v). \label{eq-prf-exp-p1}
\end{align}
For $(\bar i,\bar u) \not=(s, \tilde u)$, Eq.~(\ref{eq-prf-exp-p1}) is $ p^{\ell,t+1}_1(\cdot \mid \bar i, \bar u) =  p^{\ell,t}_1( \cdot \mid \bar i, \bar u)$; since  $p^{\ell,t}_1( \cdot \mid \bar i, \bar u) \in \A_\delta \big(\mbf{p}^{\bar i \bar u}_{\bplb} \big)$ by property (b) of the induction hypothesis, to prove Eq.~(\ref{eq-prf-p1-propb}), we only have one case left to consider: $(\bar i,\bar u) = (s, \tilde u)$ and $p^{\ell,t+1}_1(\cdot \mid s, \tilde u)$ is given by Eq.~(\ref{eq-prf-exp-p1}). 
Now if $s = 0$, then $(\bar i,\bar u) = (s, \tilde u) = (0,0)$ and we have $p^{\ell,t+1}_1(0 \mid 0, 0) = 1$ as desired, because $p^{\ell,t}_1( 0 \mid 0, 0)= 1$ by the induction hypothesis and $p^{\ell_{\tilde v}, \tau_{\tilde v}}_0(0 \mid 0, 0, 0) = 1$ by our definition of $p^{\ell_{\tilde v}, \tau_{\tilde v}}_0$ for $s = 0$. So consider the case $(\bar i,\bar u) = (s, \tilde u) \not= (0,0)$.
By property (b) of the induction hypothesis, 
\begin{equation}
p^{\ell_{\tilde v}, \tau_{\tilde v}}_0(\cdot \mid s, \tilde u, \tilde v) \in \A_\delta \big(\mbf{p}^{\ell_{\tilde v}}_o \big), \qquad \tilde v \in V(s). \label{eq-prf-exp-p1a}
\end{equation}
In view of Eq.~(\ref{eq-sspa-transprob}) in the definition of \sspa\ [Definition~\ref{def-sspa}],  
$\mbf{p}^{s\tilde u}_{\bplb} =   \sum_{\tilde v \in V(s)}   \bplb_s(\tilde v) \,  \mbf{p}^{\ell_{\tilde v}}_o$, and therefore,
the relation~(\ref{eq-prf-exp-p1a}) implies that
$$ \sum_{\tilde v \in V(s)} \bplb_s(\tilde v) \, p^{\ell_{\tilde v}, \tau_{\tilde v}}_0(\cdot \mid s, \tilde u, \tilde v) \, \in \A_\delta \big(\mbf{p}^{s\tilde u}_{\bplb} \big).$$
Using this fact and the induction hypothesis that $p^{\ell,t}_1( \cdot \mid s, \tilde u) \in \A_\delta \big(\mbf{p}^{s \tilde u}_{\bplb} \big)$, we obtain from the convex combination formula (\ref{eq-prf-exp-p1}) that
$p^{\ell,t+1}_1(\cdot \mid s, \tilde u) \in \A_\delta \big(\mbf{p}^{s \tilde u}_{\bplb} \big)$.
This proves Eq.~(\ref{eq-prf-p1-propb}) and shows that $p^{\ell,t+1}_1$ satisfies the requirement in property (b) for $t+1$.

\bigskip
\noindent {\bf Define the Markov chain for $t+1$:}
\smallskip

We now define the time-inhomogeneous Markov chain $(i_0, u_0, v_0), (i_1, u_1), (i_2, u_2), \ldots$ with probability distribution $\mbf{P}^{\ell}_{t+1}$, as required in property (a) for $t+1$.
Let the chain start with 
$(i_0, u_0, v_0) = (i, u, v)$, and let its transition probabilities have the product forms given in property (b) for $t+1$, where 
$p^{\ell,t+1}_k, k \geq 0,$ and $\mu^{\ell,t+1}_k, k \geq 1,$ are the functions that we defined in the preceding proof.
Also let the time-varying one-stage cost functions $g^{\ell,t+1}_k, k \geq 0,$ be as defined earlier.
We have shown that these transition probabilities and one-stage cost functions satisfy the requirements in properties (b)-(d). 
To prove the lemma, what we still need to show is that with our definitions,  the expression given in property (a) equals $\tilde Q_{t+1}(i,u,v)$.

First of all, our definitions of the transition probabilities and one-stage cost functions for time $t+1$ ensure that $\{ (i_k, u_k), k \geq k_{t+1} \}$ evolves and incurs costs as in \sspa\ under the proper policy $\mu$ [property (d)]. Consequently, 
$\E^{ \mbf{P}^{\ell}_{t+1}} \Big[ \sum_{k=1}^\infty  g^{\ell, t+1}_k(i_k, u_k) \Big] $
is well-defined and finite, and the order of summation and expectation can be exchanged:
\begin{equation}  
  \E^{ \mbf{P}^{\ell}_{t+1}} \Big[ \, \sum_{k=1}^\infty  g^{\ell, t+1}_k(i_k, u_k)  \, 
  \Big] =
  \sum_{k=1}^\infty \E^{ \mbf{P}^{\ell}_{t+1}} \Big[ \,  g^{\ell, t+1}_k(i_k, u_k)  \, 
  \Big]. \notag
\end{equation}  
Now $ \tilde Q_{t+1}(i,u,v) = \sum_{k = 0}^\infty C_k$ by Eq.~(\ref{eq-prf-tQ}). Hence, to prove property (a) for $t+1$, that is, to show 
$$ \tilde Q_{t+1}(i,u,v) = g^{\ell,t+1}_0(i, u, v)  + \sum_{k=1}^\infty  \E^{ \mbf{P}^{\ell}_{t+1}} \Big[ \,  g^{\ell,t+1}_k(i_k, u_k) \, \Big],$$
we only need to show that
\begin{equation} \label{eq-prf-ck}
   C_0 =  g^{\ell,t+1}_0(i, u, v), \qquad C_k =  \E^{ \mbf{P}^{\ell}_{t+1}} \Big[ \,  g^{\ell,t+1}_k(i_k, u_k) \, \Big], \quad k \geq 1.
\end{equation}   
The equality for $C_0$ above is true since by definition $g^{\ell, t+1}_0(i, u, v) = \G_{t+1}(i,u,v) = C_0$ [cf.\ Eq.~(\ref{eq-prftQ-1})].
We now prove the second equality in Eq.~(\ref{eq-prf-ck}) for $C_k, k \geq 1$.

For $k \geq 1$, recall
$$C_k = \sum_{\bar i \in \Ste} \sum_{  \bar u \in U(\bar i)}  \, P^k_3(X = \bar i, Y = \bar u) \,\cdot g^{\ell, t+1}_k(\bar i, \bar u)$$ 
[cf.\ Eqs.~(\ref{eq-prf-exptQk1}),~(\ref{eq-prf-expC1})]. Hence, to show the desired equality (\ref{eq-prf-ck}) for $C_k$, it is sufficient to show that
\begin{equation} \label{eq-prf-final-marg1}
  \mbf{P}^{\ell}_{t+1}( i_k = \bar i, u_k = \bar u) = P^k_3(X = \bar i, Y = \bar u),  \qquad \forall \, (\bar i, \bar u) \in \Stce^1.
\end{equation}  
By the definition of $\mbf{P}^{\ell}_{t+1}$ [which is defined by property (b), as we recall], $\mbf{P}^{\ell}_{t+1}( u_k = \bar u \mid i_k = \bar i) = \mu_k^{\ell,t+1}(\bar u \mid \bar i)$ for all $(\bar i, \bar u) \in \Stce^1$, 
so in view of Eq.~(\ref{eq-prf-nu}) (which is the defining relation for $\mu_k^{\ell,t+1}$),  
the equality (\ref{eq-prf-final-marg1}) will be implied if we show 
\begin{equation} \label{eq-prf-final-marg2}
   \mbf{P}^{\ell}_{t+1}( i_k = \bar i) = P^k_3(X = \bar i), \qquad \forall \, \bar i \in \Ste.
\end{equation}   

We verify Eq.~(\ref{eq-prf-final-marg2}) by induction on $k$. For $k = 1$, from Eq.~(\ref{eq-prf-P13}) and property (b) of $\mbf{P}^{\ell}_t$, we have that for every $\bar i \in \Ste$,
\begin{align*}
  P^1_3(X = \bar i) & =  ( 1 - \gamma ) \, \mbf{P}^{\ell}_t \big(i_1 = \bar i  
   \big) + \gamma \,  \mbf{e}_s(\bar i)  \\
   & = ( 1 - \gamma ) \, p^{\ell, t}_0(\bar i \mid i, u, v) + \gamma \, \mbf{e}_s(\bar i) \\
   & = ( 1 - \gamma ) \, \mbf{q}^{\ell}_t(\bar i) + \gamma \, \mbf{e}_{j^{\ell}_t}(\bar i)   \\
   &  = \mbf{q}^{\ell}_{t+1}(\bar i) 
   = p^{\ell,t+1}_0(\bar i \mid i, u, v) = \mbf{P}^{\ell}_{t+1}( i_1 = \bar i ),
\end{align*}   
where the last three equalities follow from the definition of  $\mbf{q}^{\ell}_{t+1}$ [cf.\ Eq.~(\ref{eq-definitp})], the definition of $p^{\ell,t+1}_0$ [Eq.~(\ref{eq-prf-def-gp0})],  and the definition of $\mbf{P}^{\ell}_{t+1}$, respectively.
Hence Eq.~(\ref{eq-prf-final-marg2}) holds for $k = 1$.

Suppose Eq.~(\ref{eq-prf-final-marg2}) holds for some $k \geq 1$. Then, by the definition of $\mbf{P}^{\ell}_{t+1}$ [i.e., the property (b)],  we have for all $\bar j \in \Ste$,
\begin{align*}
   \mbf{P}^{\ell}_{t+1}( i_{k+1} = \bar j) & =  \sum_{\bar i \in \Ste} \sum_{\bar u \in U(\bar i)} \mbf{P}^{\ell}_{t+1}( i_{k} = \bar i) \cdot \mu^{\ell,t+1}_k(\bar u \mid \bar i ) \cdot p^{\ell,t+1}_k(\bar j \mid \bar i, \bar u) \\
   & = \sum_{\bar i \in \Ste} \sum_{\bar u \in U(\bar i)} P^k_3( X = \bar i) \cdot \mu^{\ell,t+1}_k(\bar u \mid \bar i ) \cdot p^{\ell,t+1}_k(\bar j \mid \bar i, \bar u) \\
 & = P^k_3(Z = \bar j)  
 = P^{k+1}_3(X = \bar j),
 \end{align*}
where the second equality follows from the induction hypothesis, the third equality follows from Eqs.~(\ref{eq-prf-nu})-(\ref{eq-prf-tp}), and the last equality follows from the definition of $P^k_3$ and $P^{k+1}_3$.
This completes the induction and proves that Eq.~(\ref{eq-prf-final-marg2}) holds for all $k \geq 1$, which in turn proves that Eq.~(\ref{eq-prf-final-marg1}) holds for all $k \geq 1$.
Consequently, for all $k \geq 1$, the desired equality (\ref{eq-prf-ck}) for $C_k$ holds.
This completes the proof of Lemma~\ref{lma-tQ}.

\end{document}